\documentclass{amsart}

\usepackage{amsmath}
\usepackage{amsfonts}
\usepackage{amssymb}
\usepackage{pdfsync}
\usepackage{color,graphicx}
\usepackage{tikz}
\usepackage{a4wide}
\usepackage{amscd,amsthm,latexsym}
\usepackage{hyperref}
\numberwithin{equation}{section}

\newtheorem{thm}{Theorem}[section]
\newtheorem{lem}[thm]{Lemma}
\newtheorem{prop}[thm]{Proposition}
\newtheorem{cor}[thm]{Corollary}
\theoremstyle{definition}

\newtheorem{defn}[thm]{Definition}

\newtheorem{remark}[thm]{Remark}

\newcommand\ZZ{\mathbb{Z}}
\newcommand\CC{\mathbb{C}}

\newcommand\NW{\operatorname{NW}}
\newcommand\NE{\operatorname{NE}}

\newcommand\LHT{{\operatorname{LHT}}}

\newcommand\Par{{\mathcal{P}}}

\newcommand\lm{{\lambda/\mu}}
\newcommand\wt{\operatorname{wt}}
\renewcommand\vec[1]{\mathbf{#1}}

\newcommand\flr[1]{\left\lfloor #1\right\rfloor}
\newcommand\ceiling[1]{\left\lceil #1\right\rceil}
\newcommand\Qbinom[3]{\genfrac{[}{]}{0pt}{}{#1}{#2}_{#3}}
\newcommand\qbinom[2]{\Qbinom{#1}{#2}{q}}
\newcommand\AL{{AL}}

\newcommand\LL{\mathcal{L}}
\newcommand\LLL{\mathfrak{L}}

\newcommand\qhyper[5]{
  {}_{#1}\phi_{#2} \left( 
    \begin{matrix}
      #3\\
      #4\\
    \end{matrix}
    ; #5
    \right)
}

\newcommand\cell[3]{
\def\i{#1} \def\j{#2} \def\entry{#3}
\draw (\j-1,-\i)--(\j,-\i)--(\j,-\i+1);
\node at (\j-.5,-\i+.5) {\entry};
}
\newcommand\LHL[2]{
  \def\X{#1} \def\Y{#2}
  \node at (-.3,\Y+1.1) {$\infty$};
  \foreach \i in {0,...,\X}
  {
    \draw[gray,very thin] (\i,0) -- (\i,\Y+1.1);
    \node at (\i,-.3) {\i};
  }
  \foreach \j in {0,...,\Y}
  {
    \node at (-.3,\j) {\j};
  }
  \foreach \x in {1,...,\X}
  {
    \pgfmathsetmacro{\m}{\Y*\x}
    \foreach \y in {0,...,\m}
    {
      \draw[gray,very thin] (\x-1,\y/\x) -- (\x,\y/\x);
    }
  }
}
\newcommand\SLHL[2]{
  \def\X{#1} \def\Y{#2}
  \node at (-.3,\Y+1.1) {$\infty$};
  \foreach \i in {0,...,\X} 
  {
    \pgfmathsetmacro{\b}{-1/(\i+1)^2}
    \draw[gray,thin] (\i,\b) -- (\i,\Y+1.1); 
    \node at (\i,\b-.3) {$\i$};                      
  }
  \foreach \j in {0,...,\Y}
  {
    \node at (-.3,\j) {\j};     
  }
  \foreach \x in {1,...,\X}
  {
    \pgfmathsetmacro{\m}{\Y*\x}
    \foreach \y in {0,...,\m}
    {
      \pgfmathsetmacro{\b}{\y/\x-1/(\x)^2}
      \draw[gray,very thin] (\x-1,\b) -- (\x,\y/\x);
    }
  }
}

\title{Lecture hall tableaux}

\author{Sylvie Corteel and Jang Soo Kim}
\address{Department of Mathematics, UC Berkeley, United States}
\email{corteel@berkeley.edu}

\address{Department of Mathematics, Sungkyunkwan University (SKKU), South Korea}
\email{jangsookim@skku.edu}

\date{\today}



\begin{document}

\begin{abstract}
 We introduce lecture hall tableaux, which are fillings of a skew Young diagram satisfying certain conditions.  Lecture hall tableaux generalize both lecture hall partitions and anti-lecture hall compositions, and also contain reverse semistandard Young tableaux as a limit case.  We show that the coefficients in the Schur expansion of multivariate little $q$-Jacobi polynomials are generating functions for lecture hall tableaux. Using a Selberg-type integral we show that  moments of multivariate little $q$-Jacobi polynomials, which are equal to generating functions for lecture hall tableaux of a Young diagram, have a product formula.  We also explore various combinatorial properties of lecture hall tableaux.
\end{abstract}

\maketitle

\section{Introduction}

A \emph{lecture hall partition} is a sequence
$\lambda=(\lambda_1,\dots,\lambda_n)$ of
integers satisfying
\[
\frac{\lambda_1}{n}\ge \frac{\lambda_2}{n-1}\ge \dots \ge \frac{\lambda_n}{1}\ge 0  
\]
and an \emph{anti-lecture hall composition} is a sequence
$\alpha=(\alpha_1,\dots,\alpha_n)$ of integers satisfying
\[
\frac{\alpha_1}{1}\ge \frac{\alpha_2}{2}\ge \dots \ge \frac{\alpha_n}{n}\ge 0.     
\]
Lecture hall partitions were introduced by Bousquet-M\'elou and Eriksson
\cite{BME1,BME2} who showed that the generating function for them has a simple
product formula and has a close connection with the Poincar\'e polynomial of the
affine Coxeter group of type $\widetilde{C}_n$ using results of Eriksson and
Eriksson \cite{Eriksson_1998}. Corteel and Savage \cite{ALHC} introduced
anti-lecture hall compositions and showed that they also have a simple
generating function. Lecture hall partitions and anti-lecture hall compositions
have been studied extensively in the last two decades, see the recent survey
written by Savage \cite{LHPSavage} and references therein. In this paper we show
that these objects are closely related to the little $q$-Jacobi polynomials
$p^L_n(x;a,b;q)$ and we generalize them to 2-dimensional arrays called lecture
hall tableaux.

For monic univariate orthogonal polynomials $p_n(x)$ with linear functional
$\LL$, the \emph{mixed moment} $\sigma_{n,k}$ and the \emph{(normalized) moment}
$\sigma_n$ are defined by
\[
\sigma_{n,k}= \frac{\LL(x^n p_k(x))}{\LL(p_k(x)^2)}, \qquad \sigma_n = \sigma_{n,0} = \frac{\LL(x^n)}{\LL(1)}.
\]
See for example \cite{CKS, Viennot} for surveys on moments of
orthogonal polynomials. Since $\sigma_{n,k}$ is the coefficient $[p_k(x)]x^n$ of
$p_k(x)$ in $x^n$, it is natural to define the \emph{dual mixed moment}
$\nu_{n,k}$ to be the coefficient $[x^k]p_n(x)$ of $x^k$ in $p_n(x)$. In other
words, the mixed moments $\sigma_{n,k}$ and the dual mixed moments $\nu_{n,k}$
satisfy
\[
x^n = \sum_{k=0}^n \sigma_{n,k} p_k(x),\qquad
p_n(x) = \sum_{k=0}^n \nu_{n,k} x^k.
\]

In this paper we show that the mixed moments and the dual mixed moments of the
little $q$-Jacobi polynomials are generating functions for anti-lecture hall
compositions and lecture hall partitions, respectively. We then extend this
result to the multivariate little $q$-Jacobi polynomials.

A \emph{partition} is a sequence $\lambda=(\lambda_1,\dots,\lambda_n)$ of nonnegative integers $\lambda_1\ge\lambda_2\ge\dots\ge \lambda_n\ge0$. 
If $\lambda_i>0$, we say that $\lambda_i$ is a \emph{part} of $\lambda$. The number of parts of $\lambda$ is denoted by $\ell(\lambda)$. If $\ell(\lambda)=k$, we use the convention that $\lambda_i=0$ for all $i>k$. Let $\Par_n$ denote the set of partitions with at most $n$ parts. 
For any sequence $\alpha=(\alpha_1,\dots,\alpha_n)$ of integers, we denote by 
$|\alpha|$ the sum $\alpha_1+\dots+\alpha_n$. 

In many cases, a family $\{p_n(x)\}_{n\ge0}$ of univariate orthogonal polynomials generalizes naturally 
to a family $\{p_\lambda(x_1,\dots,x_n)\}_{\lambda\in\Par_n}$ of multivariate orthogonal polynomials via
\begin{equation}
  \label{eq:p_lambda}
p_\lambda(x_1,\dots,x_n) =
\frac{\det\left( p_{\lambda_j+n-j}(x_i)\right)_{i,j=1}^n}{\Delta(x)},
\end{equation}
where
\[
\Delta(x)=\Delta (x_1,\dots,x_n) = \prod_{1\le i<j\le n} (x_i-x_j).
\]
See for example \cite{Stokman97}.
Note that the \emph{Schur function} $s_\lambda(x)$ is also constructed in this way using the basis $\{x^n\}_{n\ge0}$:
\[
s_\lambda(x_1,\dots,x_n) =
\frac{\det\left(x_i^{\lambda_j+n-j}\right)_{i,j=1}^n}{\Delta(x)}.
\]

Suppose that $p_\lambda(x_1,\dots,x_n)$ are multivariate orthogonal polynomials given by \eqref{eq:p_lambda} with orthogonality
\[
  \mathfrak{L}_n(p_\lambda(x_1,\dots,x_n) p_\mu(x_1,\dots,x_n)) =
  \delta_{\lambda,\mu} K_\lambda(n),
\]
where $\delta_{\lambda,\mu}$ is the Kronecker delta, $\mathfrak{L}_n$ is a
linear functional on the space of multivariate polynomials in variables
$x_1,\dots,x_n$, and $K_\lambda(n)$ is a function depending on $\lambda$ and
$n$. Considering $s_\lambda(x_1,\dots,x_n)$ as a multivariate analog of $x^i$,
we define the \emph{mixed moment} $M_{\lambda,\mu}(n)$ of
$\{p_\lambda(x_1,\dots,x_n)\}_{\lambda\in\Par_n}$ by
\[
M_{\lambda,\mu}(n) = \frac{\mathfrak{L}_n (s_\lambda(x_1,\dots,x_n) p_\mu(x_1,\dots,x_n))}{\mathfrak{L}_n (p_\mu(x_1,\dots,x_n)^2)},
\]
and the \emph{moment} $M_\lambda(n)$ of $\{p_\lambda(x_1,\dots,x_n)\}_{\lambda\in\Par_n}$ by
\[
M_{\lambda}(n)= M_{\lambda,\emptyset}(n) = \frac{\mathfrak{L}_n (s_\lambda(x_1,\dots,x_n))}{\mathfrak{L}_n (1)}.
\]
If the univariate orthogonal polynomials are the Askey--Wilson polynomials, the corresponding multivariate polynomials are the Koornwinder polynomials with $q=t$. In this case the moments appear naturally in connection with exclusion processes \cite{CMW,CW_Koor}.

Similarly to the univariate case, the mixed moment $M_{\lambda,\mu}(n)$ is the coefficient of $p_\mu$ in the expansion of $s_\lambda$:
\[
s_\lambda(x_1,\dots,x_n) = \sum_{\mu\in\Par_n} M_{\lambda,\mu}(n) p_\mu(x_1,\dots,x_n).
\]
We define the \emph{dual mixed moment} $N_{\lambda,\mu}(n)$ by
\[
p_\lambda(x_1,\dots,x_n) = \sum_{\mu\in\Par_n} N_{\lambda,\mu}(n) s_\mu(x_1,\dots,x_n).
\]

The \emph{multivariate little $q$-Jacobi polynomials}
$p^L_\lambda(x_1,\dots,x_n;a,b;q)$ are defined by the equation
\eqref{eq:p_lambda} using $p^L_n(x;a,b;q)$. It is known that
$p^L_\lambda(x_1,\dots,x_n;a,b;q)$ are multivariate orthogonal polynomials with
explicit linear functional $\mathfrak{L}^L$ related to the $q$-Selberg integral
\cite{Stokman97}. Therefore, we can consider their mixed moments
$M^L_{\lambda,\mu}(n;a,b)$ and the dual mixed moments
$N^L_{\lambda,\mu}(n;a,b)$. In this paper we give a combinatorial interpretation
for these quantities using new combinatorial objects called lecture hall
tableaux.

Let $\lambda\in\Par_n$. The \emph{Young diagram} of $\lambda$ is the set
$\{(i,j): 1\le i\le \ell(\lambda), 1\le j\le \lambda_i\}$. We identify $\lambda$
with its Young diagram. Each element $(i,j)\in\lambda$ is called a \emph{cell}.
For a cell $(i,j)$ in $\lambda$, the \emph{content} $c(i,j)$ is defined by
$c(i,j)=j-i$. The notation $\mu\subseteq\lambda$ means the Young diagram
containment. If $\mu\subseteq\lambda$, the \emph{skew Young diagram} $\lm$ is
defined to be the set-theoretic difference $\lambda-\mu$ of their Young
diagrams. We will draw $\lm$ by placing a square in row $i$ and column $j$ for
each $(i,j)\in\lm$. A Young diagram $\lambda$ is also considered as the skew
Young diagram $\lambda/\emptyset$.

\begin{defn}
 For an integer $n$ and  partitions $\lambda$ and $\mu$ with $\mu\subseteq\lambda$ and $\ell(\lambda)\le n$, a \emph{lecture hall tableau} of shape $\lm$ 
and of type $(n,\ge,>)$  is a filling $T$ of the cells in the Young diagram $\lm$ with nonnegative integers satisfying the following conditions:
\[
\frac{T(i,j)}{n+c(i,j)}  \ge \frac{T(i,j+1)}{n+c(i,j+1)}, \qquad
\frac{T(i,j)}{n+c(i,j)} > \frac{T(i+1,j)}{n+c(i+1,j)}.
\]
We denote by $\LHT_{(n,\ge,>)}(\lm)$ the set of such fillings. Similarly we
define $\LHT_{(n,<,\le)}(\lm)$ with the inequalities $\ge$ and $>$ replaced by
$<$ and $\le$ respectively.
\end{defn}

See Figure~\ref{fig:LHT} for an example of a lecture hall tableau. If $\lm$ has
only one row (resp.~column), the lecture hall tableaux in
$\LHT_{(n,\ge,>)}(\lm)$ become anti-lecture hall compositions (resp.~lecture
hall partitions). A \emph{reverse semistandard Young tableau} of shape $\lm$ is
a filling of $\lm$ with nonnegative integers such that the entries are weakly
decreasing in each row and strictly decreasing in each column. Lecture hall
tableaux also generalize reverse semistandard Young tableaux in the sense that
if $n\to\infty$, lecture hall tableaux of type $(n,\ge,>)$ become reverse
semistandard Young tableaux. Moreover, the lecture hall tableaux in
$\LHT_{(n,\ge,>)}(\lambda)$ whose entries are at most $n$ are exactly the
reverse semistandard Young tableaux of shape $\lambda$ whose entries are at most
$n$. Recently, Br\"and\'en and Leander \cite{BL} generalized lecture hall
partitions to \emph{lecture hall $P$-partitions} in a similar way that partitions are
generalized to $P$-partitions. Lecture hall tableaux are a special case of
lecture hall $P$-partitions.

Consider a sequence $\vec y = (y_0,y_1,\dots)$ of variables. For $T$ in $LS^{(n,\ge,>)}_{\lm}$ or $LS^{(n,<,\le)}_{\lm}$, the \emph{weight} $\wt(T)$ is defined by
\[
\wt(T)=\prod_{s\in \lm} y_{T(s)} u^{\flr{T(s)/(n+c(s))}} v^{o(\flr{T(s)/(n+c(s))})},
\]
where $o(m)$ is $1$ if $m$ is odd and $0$ otherwise. 
For example, if $T$ is the lecture hall tableau in Figure~\ref{fig:LHT}, its weight is
\[
\wt(T)=y_0^3y_1^3y_2^2y_3^2y_4^2y_5y_6y_9 u^3 v^3.
\]

\begin{figure}
  \centering
\begin{tikzpicture}[scale=.6]
\cell149 \cell154 \cell163
\cell225 \cell236 \cell244 \cell253 \cell261
\cell312 \cell322 \cell331 \cell340
\cell411 \cell420 \cell430
\draw (0,-4)--(0,-2)--(1,-2)--(1,-1)--(3,-1)--(3,0)--(6,0);
\end{tikzpicture} \qquad \qquad
\begin{tikzpicture}[scale=.6]
\cell14{$\frac98$} \cell15{$\frac49$} \cell16{$\frac{3}{10}$}
\cell22{$\frac55$} \cell23{$\frac66$} \cell24{$\frac47$} \cell25{$\frac38$} \cell26{$\frac19$}
\cell31{$\frac23$} \cell32{$\frac24$} \cell33{$\frac15$} \cell34{$\frac06$}
\cell41{$\frac12$} \cell42{$\frac03$} \cell43{$\frac04$}
\draw (0,-4)--(0,-2)--(1,-2)--(1,-1)--(3,-1)--(3,0)--(6,0);
\end{tikzpicture}
  \caption{On the left is a lecture hall tableau $T\in \LHT_{(n,\ge,>)}(\lm)$ for $n=5$, $\lambda=(6,6,4,3)$ and $\mu=(3,1)$. 
The diagram on the right shows the number $T(i,j)/(n+c(i,j))$ for each entry $(i,j)\in\lm$.}
  \label{fig:LHT}
\end{figure}

We define the \emph{lecture hall Schur functions} of shape $\lm$ and of types $(n,\ge,>)$
and $(n,<,\le)$ by
\begin{align*}
LS^{(n,\ge,>)}_{\lm}(\vec y;u,v) &= \sum_{T\in \LHT_{(n,\ge,>)}(\lm)} \wt(T),\\
LS^{(n,<,\le)}_{\lm}(\vec y;u,v) &= \sum_{T\in \LHT_{(n,<,\le)}(\lm)} \wt(T).  
\end{align*}
These lecture hall Schur functions become the usual Schur functions when $n\to\infty$:
\[
\lim_{n\to\infty} LS_\lm^{(n,\ge,>)}(\vec y;u,v) = s_{\lm}(\vec y),\qquad
\lim_{n\to\infty} LS_\lm^{(n,<,\le)}(\vec y;u,v) = s_{\lambda'/\mu'}(\vec y),
\]
where $\lambda'$ is the \emph{conjugate} of $\lambda$, i.e., in terms of Young
diagrams $\lambda'=\{(j,i):(i,j)\in\lambda\}$. We show that they also have
Jacobi--Trudi type formulas, see Theorems~\ref{thm:schur} and \ref{thm:schur2}.
These formulas show that lecture hall Schur functions are a special case of
Macdonald's 9th variation of Schur functions \cite{Macdonald_Schur}, which
generalize several variations of Schur functions. Note that lecture hall Schur 
functions are not symmetric in the variables $\vec y$.

Let $\vec q = (1,q,q^2,\dots)$ be the \emph{principal specialization}
$y_i=q^{i}$ of $\vec y=(y_0,y_1,y_2,\dots)$. In this paper we show that the
mixed moments $M^L_{\lambda,\mu}(n;a,b;q)$ and the dual mixed moments
$N^L_{\lambda,\mu}(n;a,b;q)$ for the multivariate little $q$-Jacobi polynomials
$p^L_\lambda(x_1,\dots,x_n;a,b;q)$ are generating functions for lecture hall
tableaux.

\begin{thm}\label{thm:main1}
We have
\begin{align*}
N^L_{\lambda,\mu}(n;-uv,-u/v;q)&=(-1)^{|\lm|} LS^{(n,<,\le)}_{\lm}(\vec q;u,v),\\  
M^L_{\lambda,\mu}(n;-uv,-u/v;q)&=LS^{(n,\ge,>)}_{\lm}(\vec q;u,v).
\end{align*}
Equivalently,
\begin{align*}
p^L_\lambda(x_1,\dots,x_n;-uv,-u/v;q) 
&= \sum_{\mu\subseteq\lambda} (-1)^{|\lm|} LS^{(n,<,\le)}_{\lm}(\vec q;u,v)s_\mu(x_1,\dots,x_n),\\
s_\lambda(x_1,\dots,x_n) 
&= \sum_{\mu\subseteq\lambda} LS^{(n,\ge,>)}_{\lm}(\vec q;u,v)p^L_\mu(x_1,\dots,x_n;-uv,-u/v;q).
\end{align*}
\end{thm}
 
Note that the \emph{moments}
$M^L_{\lambda}(n;a,b;q):=M^L_{\lambda,\emptyset}(n;a,b;q)$ and the \emph{dual
  moments} $N^L_{\lambda}(n;a,b;q):=N^L_{\lambda,\emptyset}(n;a,b;q)$ are the
generating functions for lecture hall tableaux of a straight shape
$\lambda=\lambda/\emptyset$. We prove the following theorem, which shows that
the moments and the dual moments have product formulas. Throughout this paper we
use the standard notation for $q$-series:
\[
(a)_n=(a;q)_n = (1-a)(1-aq)\cdots(1-aq^{n-1}).
\]
\begin{thm}\label{thm:main}
Given an integer $n$ and a partition $\lambda$ with at most $n$ parts,
\begin{align*}
LS^{(n,\ge,>)}_{\lambda}(\vec q;u,v) &=
\prod_{1\le i<j\le n} \frac{q^{\lambda_j+n-j}-q^{\lambda_i+n-i}}{q^{i-1}-q^{j-1}}
\prod_{i=1}^n \frac{(-uvq^{n-i+1};q)_{\lambda_i}}{(u^2q^{2n-i+1};q)_{\lambda_i}}, \\
LS^{(n,<,\le )}_{\lambda}(\vec q;u,v)&= q^{n(\lambda)-n(\lambda')} \prod_{1\le i<j\le n} \frac{q^{\lambda_j+n-j}-q^{\lambda_i+n-i}}{q^{i-1}-q^{j-1}} 
\prod_{i=1}^n \frac{(-uvq^{n-i+1};q)_{\lambda_i}}{(u^2q^{n-i+1+\lambda_i};q)_{n-i+\lambda_i}}\\
 \notag & \qquad  \times \prod_{1\le i<j\le n}(1-u^2q^{2n+\lambda_i+\lambda_j-i-j+1}),
\end{align*}
where  $n(\lambda)=\sum_{i=1}^{\ell(\lambda)} (i-1)\lambda_i$.
\end{thm}

The rest of this paper is organized as follows.  In Section~\ref{sec:defin-basic-results} we recall lecture hall partitions, anti-lecture hall compositions and their basic properties. We show that the mixed moments and the dual mixed moments of the little $q$-Jacobi polynomials are generating functions for lecture hall partitions and anti-lecture hall compositions respectively. We also find a connection between the mixed and dual mixed moments of univariate orthogonal polynomials and those of corresponding multivariate orthogonal polynomials.  In Section~\ref{sec:lect-hall-tabl}, we introduce lecture hall tableaux and their multivariate generating functions called lecture hall Schur functions. Using lattice path models we prove Jacobi--Trudi type formulas for lecture hall Schur functions.  In Section~\ref{sec:mult-little-q}, we prove Theorem~\ref{thm:main1} using the results in Section~\ref{sec:lect-hall-tabl}.  In Section~\ref{sec:moments-dual-moments}, we prove Theorem~\ref{thm:main} using a $q$-Selberg integral for the first identity and determinant evaluations for the second identity.  In Section~\ref{sec:q-enum-lect}, we consider two other families of lecture hall tableaux and prove similar enumeration results.  In Section~\ref{sec:concl-furth-plans}, we propose open problems and future generalizations.

\section{Definitions and basic results}
\label{sec:defin-basic-results}

In this section we define lecture hall partitions, anti-lecture hall
compositions and the little $q$-Jacobi polynomials. We present basic results of
these objects and multivariate orthogonal polynomials which will be used in later sections.

\subsection{Lecture hall partitions and anti-lecture hall compositions}

A \emph{lecture hall partition} is a sequence
$\lambda=(\lambda_1,\dots,\lambda_n)$ of
integers satisfying
\[
\frac{\lambda_1}{n}\ge \frac{\lambda_2}{n-1}\ge \dots \ge \frac{\lambda_n}{1}\ge 0.  
\]
Let $L_n$ denote the set of lecture hall partitions
$\lambda=(\lambda_1,\dots,\lambda_n)$ satisfying the above conditions.
Bousquet-M\'elou and Eriksson \cite{BME1} introduced lecture hall partitions and
showed that
\[
\sum_{\lambda\in L_n} q^{|\lambda|} = \frac{1}{(q;q^2)_n}.
\]
See \cite{LHPSavage} for the origin of lecture hall partitions and their connections with many other objects. 

An \emph{anti-lecture hall composition}  (or a \emph{planetarium composition}) is a sequence
$\alpha=(\alpha_1,\dots,\alpha_n)$ of integers satisfying
\[
\frac{\alpha_1}{1}\ge \frac{\alpha_2}{2}\ge \dots \ge \frac{\alpha_n}{n}\ge 0.     
\]
Let $AL_n$ denote the set of anti-lecture hall compositions $\alpha=(\alpha_1,\dots,\alpha_n)$ satisfying the above condition.
Corteel and Savage \cite{ALHC} showed that
\[
\sum_{\alpha\in AL_n} q^{|\alpha|} = \frac{(-q;q)_n}{(q^2;q)_n}.
\]

Now we consider truncated versions of lecture hall partitions and anti-lecture compositions.

\begin{defn}
  For integers $n\ge k\ge0$, we define
  \begin{align*}
L_{n,k} &=  \left\{(\lambda_1,\dots,\lambda_k)\in \ZZ^k: 
\frac{\lambda_1}{n}> \frac{\lambda_2}{n-1}>\dots > \frac{\lambda_k}{n-k+1}\ge 0\right\},\\
\overline{L}_{n,k} &=  \left\{(\lambda_1,\dots,\lambda_k)\in \ZZ^k: 
\frac{\lambda_1}{n}\ge \frac{\lambda_2}{n-1}\ge \dots \ge \frac{\lambda_k}{n-k+1}> 0\right\},\\
AL_{n,k} &=  \left\{(\alpha_1,\dots,\alpha_k)\in \ZZ^k: 
\frac{\alpha_1}{n-k+1}\ge \frac{\alpha_2}{n-k+2}\ge \dots \ge \frac{\alpha_k}{n}\ge 0\right\},\\
\overline{AL}_{n,k} &=  \left\{(\alpha_1,\dots,\alpha_k)\in \ZZ^k: 
\frac{\alpha_1}{n-k+1}> \frac{\alpha_2}{n-k+2}>\dots > \frac{\alpha_k}{n}> 0\right\}.
  \end{align*}
\end{defn}

Note that each element in $L_{n,k}$ or $\overline{L}_{n,k}$ is obtained from a lecture hall partition in $L_n$ by truncating the last $n-k$ integers and each element in $AL_{n,k}$ or $\overline{AL}_{n,k}$ is obtained from an anti-lecture hall composition in $AL_n$ by truncating the first $n-k$ integers. In this paper we will mainly consider $L_{n,k}$ and $AL_{n,k}$. 

For a sequence $S=(s_1,\dots ,s_k)$ of positive integers and a sequence $(\beta_1,\dots ,\beta_k)$ of nonnegative integers satisfying
\[
\frac{\beta_1}{s_{1}}\ge \frac{\beta_2}{s_{2}}\ge \dots \ge \frac{\beta_k}{s_k}\ge 0,
\]
we define  
\[
\flr{\beta}_S=\left(\flr{\frac{\beta_1}{s_{1}}},
\flr{\frac{\beta_2}{s_{2}}},\dots ,\flr{\frac{\beta_k}{s_k}}\right), \qquad \ceiling{\beta}_S=\left(\ceiling{\frac{\beta_1}{s_{1}}},
\ceiling{\frac{\beta_2}{s_{2}}}, \dots , \ceiling{\frac{\beta_k}{s_k}}\right).
\]
We denote by $o(\beta)$ the number of odd integers in $\beta$. If the sequence $S$ is clear from the context, we will simply write
$\flr{\beta}$ and $\ceiling{\beta}$. For example, if $\lambda\in L_{n,k}$ or $\lambda\in \overline{L}_{n,k}$, then
$\flr{\lambda}=\flr{\lambda}_S$ and $\ceiling{\lambda}=\ceiling{\lambda}_S$ for $S=(n,n-1,\dots,n-k+1)$ and
if $\alpha\in AL_{n,k}$ or $\alpha\in \overline{AL}_{n,k}$, then $\flr{\alpha}=\flr{\alpha}_S$ and $\ceiling{\alpha}=\ceiling{\alpha}_S$ for $S=(n-k+1,n-k+2,\dots,n)$.

Now we define the following generating functions:
\begin{align*}
L_{n,k}(u,v,q) &= \sum_{\lambda \in {L}_{n,k}} u^{|\lfloor \lambda \rfloor |}v^{o(\lfloor \lambda \rfloor )} q^{|\lambda|},\\
\overline{L}_{n,k}(u,v,q) &= \sum_{\lambda \in \overline{L}_{n,k}}u^{|\lceil \lambda \rceil |}v^{o(\lceil \lambda \rceil )}  q^{|\lambda|},\\  
\AL_{n,k}(u,v,q) &= \sum_{\alpha \in \AL_{n,k}} u^{|\lfloor \alpha \rfloor |}v^{o(\lfloor \alpha \rfloor )} q^{|\alpha|},\\
\overline{AL}_{n,k}(u,v,q) &= \sum_{\alpha \in {\overline{AL}}_{n,k}} u^{|\lceil \alpha \rceil |}v^{o(\lceil \alpha \rceil )}q^{|\alpha|}.
\end{align*}
Note that the floor function $\flr{\cdot}$ is used for $L_{n,k}(u,v,q)$ and
$AL_{n,k}(u,v,q)$, whereas the ceiling function $\ceiling{\cdot}$ is used for $\overline{L}_{n,k}(u,v,q)$ and $\overline{AL}_{n,k}(u,v,q)$.

In what follows we will see that there is a simple map that gives a bijection between
$L_{n,k}$ and $\overline{L}_{n,k}$ and also a bijection between
$AL_{n,k}$ and $\overline{AL}_{n,k}$. First observe the following lemma whose proof is straightforward. 

\begin{lem}\label{lem:simple}
 For any integers $a,b$ and $i\ge1$ we have 
 \begin{itemize}
 \item $\frac{a}{i+1} > \frac{b}{i}$ if and only if $\frac{a+1}{i+1} \ge \frac{b+1}{i}$,
 \item $\frac{a}{i} \ge \frac{b}{i+1}$ if and only if $\frac{a+1}{i} > \frac{b+1}{i+1}$, 
 \item $\ceiling{\frac{a+1}{i}} = \flr{\frac{a}{i}}+1$. 
 \end{itemize}
\end{lem}

For a sequence $\beta=(\beta_1,\dots ,\beta_k)$ of integers let
$\beta^+=(\beta_1+1,\dots ,\beta_k+1)$. Then by Lemma~\ref{lem:simple} the following proposition is immediate. 

\begin{prop}\label{prop:+-}
The map $\lambda\mapsto \lambda^+$ gives a bijection from $L_{n,k}$ to $\overline{L}_{n,k}$.  
The same map $\alpha\mapsto \alpha^+$ also gives a bijection from $AL_{n,k}$ to $\overline{AL}_{n,k}$.  Moreover, 
$|\lambda^+|=|\lambda|+k$, $\ceiling{\lambda^+}=\flr{\lambda}^+$,
$|\alpha^+|=|\alpha|+k$ and $\ceiling{\alpha^+}=\flr{\alpha}^+$.
\end{prop}

By Proposition~\ref{prop:+-}, we have
\begin{align}
  \label{eq:LL'}
\overline{L}_{n,k}(u,v,q)&=(uvq)^{k} {L}_{n,k}(u,1/v,q),\\
  \label{eq:AA'}
\overline{AL}_{n,k}(u,v,q) &=(uvq)^{k}\AL_{n,k}(u,1/v,q).  
\end{align}

Corteel and Savage \cite{trunc_LHP} found product formulas
for $\overline{L}_{n,k}(u,v,q)$ and $\AL_{n,k}(u,v,q)$, see \eqref{eq:barL_nk} and \eqref{eq:AL_nk} below. Using their formulas
together with \eqref{eq:LL'}  and \eqref{eq:AA'}, we can obtain
product formulas for ${L}_{n,k}(u,v,q)$ and $\overline{\AL}_{n,k}(u,v,q)$. 
We summarize these formulas as follows.
Here we use the following notation for the $q$-binomial coefficients:
\[
\qbinom{n}{k} = \frac{(q;q)_n}{(q;q)_{k}(q;q)_{n-k}}.
\]

\begin{prop}
We have
\begin{align}
  \label{eq:L_nk}
L_{n,k}(u,v,q) &= q^{\binom k2} \qbinom{n}{k} \frac{(-uvq^{n-k+1})_k}{(u^2q^{2n-k+1})_k},\\
  \label{eq:barL_nk}  
\overline{L}_{n,k}(u,v,q) &=
 u^kv^k q^{\binom{k+1}2} \qbinom{n}{k} \frac{(-uv^{-1}q^{n-k+1})_k}{(u^2q^{2n-k+1})_k},\\
  \label{eq:AL_nk}
\AL_{n,k}(u,v,q) & = \qbinom{n}{k} \frac{(-uvq^{n-k+1})_k}{(u^2q^{2n-2k+2})_k},\\
\overline{AL}_{n,k}(u,v,q) & = (uvq)^{k} \qbinom{n}{k} \frac{(-uv^{-1}q^{n-k+1})_k}{(u^2q^{2n-k+1})_k}.
\end{align}
\end{prop}

Note that \eqref{eq:L_nk} and \eqref{eq:AL_nk} are special cases of Theorem~\ref{thm:main}.

\subsection{Mixed moments and dual mixed moments of univariate orthogonal polynomials}

Let $\{p_n(x)\}_{n\ge0}$ be a family of monic orthogonal polynomials with linear functional $\LL$.
The \emph{(normalized) moment} $\sigma_n$ of the orthogonal polynomials is defined by
\[
\sigma_n = \frac{\LL(x^n)}{\LL(1)}.
\]
The \emph{mixed moment} $\sigma_{n,k}$ of the orthogonal polynomials is defined by
\[
\sigma_{n,k} = \frac{\LL(x^np_k(x))}{\LL(p_k(x)^2)}.
\]
Note that $\sigma_n = \sigma_{n,0}$. Moments of orthogonal polynomials give rise to interesting
combinatorics \cite{CKS, Viennot}.
Since $\{p_n(x)\}_{n\ge0}$ is a basis of the polynomial space, we can write
\[
x^n = \sum_{k=0}^{n} c_{n,k} p_k(x).
\]
Multiplying $p_k(x)$ and taking $\LL$, we obtain
\[
c_{n,k} = \frac{\LL(x^np_k(x))}{\LL(p_k(x)^2)} = \sigma_{n,k} .
\]
Thus, the mixed moment $\sigma_{n,k}$ is the coefficient $[p_k(x)] x^n$ of $p_k(x)$ when we expand $x^n$ in terms of the basis $\{p_n(x)\}_{n\ge0}$.
We define the \emph{dual mixed moment} $\nu_{n,k}$ of $\{p_n(x)\}_{n\ge0}$ to be the coefficient $[x^k] p_n(x)$ of $x^k$ in $p_n(x)$.

By definition, the mixed moments $\sigma_{n,k}$ and the dual mixed moments $\nu_{n,k}$ satisfy
\[
x^n = \sum_{k=0}^n \sigma_{n,k} p_k(x),\qquad
p_n(x) = \sum_{k=0}^n \nu_{n,k} x^k.
\]
Therefore, we have
\begin{equation}
  \label{eq:inverse}
\sum_{i=0}^m \sigma_{m,i}\nu_{i,n} = \sum_{i=0}^m \nu_{m,i}\sigma_{i,n} = \delta_{m,n}.
\end{equation}
Equivalently,
\[
\big(\sigma_{i,j}\big)_{i,j=0}^\infty = \Big(\big(\nu_{i,j}\big)_{i,j=0}^\infty\Big)^{-1}.
\]

\subsection{The little $q$-Jacobi polynomials}

The \emph{(monic) little $q$-Jacobi polynomials} $p^L_n(x;a,b;q)$ are defined by
\begin{equation}
  \label{eq:jacobi_poly}
p^L_n(x;a,b;q) = \frac{(aq;q)_n}{(-1)^nq^{-\binom n2}(abq^{n+1};q)_n}
\qhyper21{q^{-n},abq^{n+1}}{aq}{q;qx},
\end{equation}
where we use the $q$-hypergeometric series notation
\[
\qhyper rs {a_1,\dots,a_r}{b_1,\dots,b_s}{q;z}
=\sum_{n=0}^\infty \frac{(a_1;q)_n\cdots (a_r;q)_n}{(q;q)_n(b_1;q)_n\cdots (b_s;q)_n}
\left((-1)^nq^{\binom n2}\right)^{1+s-r} z^n.
\]
They satisfy the following three-term recurrence relation:
\[
p^L_{n+1}(x;a,b;q) = (x-b_n)p^L_n(x;a,b;q)-\lambda_{n}p^L_{n-1}(x;a,b;q),
\]
where $b_n= A_n+C_n$, $\lambda_n=A_{n-1}C_n$ for
\begin{equation}
  \label{eq:AC}
A_n = \frac{q^n(1-aq^{n+1})(1-abq^{n+1})}{(1-abq^{2n+1})(1-abq^{2n+2})}, \qquad
C_n = \frac{aq^n(1-q^{n})(1-bq^{n})}{(1-abq^{2n})(1-abq^{2n+1})}.
\end{equation}
They are orthogonal with respect to the linear functional 
$\LL^{L,q}_{a,b}$ given by
\[
\LL^{L,q}_{a,b}(f(x))=\sum_{k\ge0} \frac{(bq;q)_k}{(q;q)_k}(aq)^k f(q^k).
\]
See \cite{KLS} for more details on the little $q$-Jacobi polynomials.

The (normalized) moment $\sigma^L_{n}(a,b;q)$ and the mixed moment  $\sigma^L_{n,k}(a,b;q)$ of the little $q$-Jacobi polynomials are given by
\[
\sigma^L_{n}(a,b;q)=\frac{\LL^{L,q}_{a,b}(x^n)}{\LL^{L,q}_{a,b}(1)},\qquad
\sigma^L_{n,k}(a,b;q)=\frac{\LL^{L,q}_{a,b}(x^n p^L_k(x;a,b;q))}{\LL^{L,q}_{a,b}(p^L_k(x;a,b;q)^2)}.
\]
Note that $\sigma^L_{n,0}(a,b;q)=\sigma^L_n(a,b;q)$ and $\sigma^L_{n,k}(a,b;q)=0$ if $n<k$.
By the $q$-binomial theorem, 
\begin{equation}
  \label{eq:jacobi_moment}
\LL^{L,q}_{a,b}(x^n) = \frac{(abq^{n+2};q)_\infty}{(aq^{n+1};q)_\infty},
\end{equation}
which gives a product formula for the moment $\sigma^L_{n}(a,b;q)$: 
\[
\sigma^L_{n}(a,b;q) =\frac{(aq;q)_{n}}{(abq^{2};q)_{n}}.
\]

The mixed moments and the dual mixed moments also have product formulas. 

\begin{lem}
We have
\begin{align}
  \label{eq:mu}
\sigma^L_{n,k}(a,b;q) &=\qbinom{n}{k} \frac{(aq^{k+1};q)_{n-k}}{(abq^{2k+2};q)_{n-k}},\\
  \label{eq:nu}
\nu^L_{n,k}(a,b;q) &= (-1)^{n-k} q^{\binom{n-k}2} \qbinom{n}{k} \frac{(aq^{k+1};q)_{n-k}}{(abq^{n+k+1};q)_{n-k}}.
\end{align}
\end{lem}
\begin{proof}
By \eqref{eq:jacobi_poly} and \eqref{eq:jacobi_moment}, we have
\begin{align*}
\LL^{L,q}_{a,b}(x^n p^L_k(x;a,b;q)) 
&= \LL^{L,q}_{a,b}\left(\frac{x^n (aq;q)_k}{(-1)^kq^{-\binom k2}(abq^{k+1};q)_k}
\qhyper21{q^{-k},abq^{k+1}}{aq}{q;qx}\right)\\
&= \frac{(aq;q)_k}{(-1)^kq^{-\binom k2}(abq^{k+1};q)_k}
\sum_{i\ge0} \frac{(q^{-k};q)_i(abq^{k+1};q)_i}{(q;q)_i(aq;q)_i} 
q^i \frac{(abq^{n+i+2};q)_\infty}{(aq^{n+i+1};q)_\infty}\\
&= \frac{(aq;q)_k}{(-1)^kq^{-\binom k2}(abq^{k+1};q)_k}
\cdot\frac{(abq^{n+2};q)_\infty}{(aq^{n+1};q)_\infty}
\qhyper32{q^{-k},abq^{k+1},aq^{n+1}}{aq, abq^{n+2}}{q;q}.
\end{align*}
By the $q$-Saalsch\"utz summation formula \cite[(II.12)]{GR}, we obtain
\begin{equation}
  \label{eq:11}
\LL^{L,q}_{a,b}(x^n p^L_k(x;a,b;q)) =\frac{a^kq^{k^2}(q^{n-k+1};q)_k(bq;q)_k (abq^{n+k+2};q)_\infty}
{(abq^{k+1};q)_k(aq^{n+1};q)_\infty}.
\end{equation}
Since $p_n^L(x;a,b;q)$ are monic, we have
\[
\LL^{L,q}_{a,b}(p^L_k(x;a,b;q)^2)=\LL^{L,q}_{a,b}(x^k p^L_k(x;a,b;q)).
\]
Thus, by  \eqref{eq:11}, we have
\[
\frac{\LL^{L,q}_{a,b}(x^n p^L_k(x;a,b;q))}{\LL^{L,q}_{a,b}(p^L_k(x;a,b;q)^2)}
=\frac{(q^{n-k+1};q)_k(abq^{n+k+2};q)_\infty(aq^{k+1};q)_\infty}
{(q;q)_k(abq^{2k+2};q)_\infty(aq^{n+1};q)_\infty}, 
\]
which is the same \eqref{eq:mu}. The second identity \eqref{eq:nu} follows from the definition
given in equation~\eqref{eq:jacobi_poly}.
\end{proof}

The following proposition shows that
the mixed moment (resp.~the dual mixed moment)
of the little $q$-Jacobi polynomials is a generating function for
anti-lecture hall compositions (resp.~lecture hall partitions).

\begin{prop}\label{prop:mu,nu}
We have
\begin{align}
\label{eq:6}
\sigma^L_{n,k}(-uv,-u/v;q) &= \AL_{n,n-k}(u,v,q),\\
\label{eq:7}
\nu^L_{n,k}(-uv,-u/v;q) &=(-1)^{n-k} L_{n,n-k}(u,v,q).
\end{align}
\end{prop}
\begin{proof}
Equation \eqref{eq:6} follows from \eqref{eq:AL_nk} and \eqref{eq:mu}. 
Equation \eqref{eq:7}  follows from \eqref{eq:L_nk} and \eqref{eq:nu}. 
\end{proof}

\begin{remark}
It is possible to give a different proof of \eqref{eq:mu} as follows. 
In \cite[Chapter~1, Proposition~17]{Viennot}, 
Viennot showed that the moment $\sigma^L_{n,k}(a,b;q)$
is the sum of weights of certain  paths in the quarter plane from $(0,0)$ to $(n,k)$. 
This interpretation gives the following recurrence for $\sigma^L_{n,k}(a,b;q)$. If $n>k$,
\[
\sigma^L_{n,k}(a,b;q)=b_k\sigma^L_{n-1,k}(a,b;q)+\lambda_{k+1}\sigma^L_{n-1,k+1}(a,b;q)+\sigma^L_{n-1,k-1}(a,b;q),
\]
where $b_n$ and $\lambda_n$ are given before \eqref{eq:AC}. Since $\sigma^L_{n,n}(a,b;q)=1$ and $\sigma^L_{n,k}(a,b;q)=0$ for $n<k$, \eqref{eq:mu} is obtained by induction. It will be interesting to find a direct combinatorial proof of \eqref{eq:6}, which is equivalent to \eqref{eq:mu}. 
\end{remark}

By \eqref{eq:inverse} and Proposition~\ref{prop:mu,nu}, we obtain the following corollary.

\begin{cor}\label{cor:LAL}
We have
  \begin{align*}
\sum_{i=0}^m \AL_{m,m-i}(u,v,q) (-1)^{i-n}L_{i,i-n}(u,v,q) &= \delta_{m,n},\\
\sum_{i=0}^m (-1)^{m-i}L_{m,m-i}(u,v,q) \AL_{i,i-n}(u,v,q) &= \delta_{m,n}.  
  \end{align*}
\end{cor}

There is a simple combinatorial proof of Corollary~\ref{cor:LAL}, see Proposition~\ref{prop:eh}.

\subsection{Multivariate orthogonal polynomials}

Suppose that $p_\lambda(x_1,\dots,x_n)$ are multivariate orthogonal polynomials with linear functional $\mathfrak{L}_n$. The \emph{mixed moment} $M_{\lambda,\mu}(n)$ 
and the \emph{(normalized) moment} $M_\lambda(n)$ of $\{p_\lambda(x_1,\dots,x_n)\}_{\lambda\in\Par_n}$ are defined by 
\[
M_{\lambda,\mu}(n) = \frac{\mathfrak{L}_n (s_\lambda(x_1,\dots,x_n) p_\mu(x_1,\dots,x_n))}{\mathfrak{L}_n (p_\mu(x_1,\dots,x_n)^2)},
\]
and
\[
M_{\lambda}(n)= M_{\lambda,\emptyset}(n) = \frac{\mathfrak{L}_n (s_\lambda(x_1,\dots,x_n))}{\mathfrak{L}_n (1)}.
\]
By the orthogonality we have
\[
s_\lambda(x_1,\dots,x_n) = \sum_{\mu\in\Par_n} M_{\lambda,\mu}(n) p_\mu(x_1,\dots,x_n).
\]
We define the \emph{dual mixed moment} $N_{\lambda,\mu}(n)$ by
\[
p_\lambda(x_1,\dots,x_n) = \sum_{\mu\in\Par_n} N_{\lambda,\mu}(n) s_\mu(x_1,\dots,x_n).
\]

We need the following well known lemma. 

\begin{lem}\label{lem:pq}
Let $\{p_n(x)\}_{n\ge0}$ and $\{q_n(x)\}_{n\ge0}$ be families of polynomials with $\deg p_n(x) = \deg q_n(x)=n$ and 
\[
p_n(x) = \sum_{k=0}^n c_{n,k} q_k(x).
\]
Then, for a partition $\lambda=(\lambda_1,\dots,\lambda_n)$, we have
\[
\det\left( p_{\lambda_j+n-j}(x_i)\right)_{i,j=1}^n
= \sum_{\mu\subseteq\lambda} \det\left(c_{\lambda_{i}+n-i,\mu_j+n-j}\right)_{i,j=1}^n 
\det\left( q_{\mu_j+n-j}(x_i)\right)_{i,j=1}^n,
\]
where $c_{i,j}=0$ if $i<j$. 
\end{lem}
\begin{proof}
Observe that
\[
\bigg( p_{\lambda_i+n-i}(x_j)\bigg)_{i,j=1}^n
=\bigg( c_{\lambda_i+n-i,k}\bigg)_{\substack{1\le i\le n\\ k\ge0}}
\bigg( q_k(x_j)\bigg)_{\substack{k\ge0 \\ 1\le j\le n}}.
\]  
By the Cauchy-Binet theorem, we have
\[
\det \bigg( p_{\lambda_i+n-i}(x_j)\bigg)_{i,j=1}^n
=\sum_{\mu_1\ge \cdots \ge \mu_n\ge 0} \det\bigg( c_{\lambda_i+n-i,\mu_j+n-j}\bigg)_{i,j=1}^n
\det\bigg( q_{\mu_i+n-i}(x_j)\bigg) _{i,j=1}^n.
\]  
Since $c_{i,j}=0$ for $i<j$, the summand vanishes unless $\mu\subseteq\lambda$, which finishes the proof.
\end{proof}

The following proposition gives a connection between the mixed and dual mixed moments
of univariate orthogonal polynomials and those of corresponding multivariate orthogonal polynomials. 

\begin{prop}
Let $\{p_i(x)\}_{i\ge0}$ be a family of univariate orthogonal polynomials. Suppose that
$\{p_\lambda(x_1,\dots,x_n)\}_{\lambda\in\Par_n}$ is a family of multivariate orthogonal polynomials defined by \eqref{eq:p_lambda}.   
Then the mixed moments $M_{\lambda,\mu}(n)$ and the dual mixed moments $N_{\lambda,\mu}(n)$ for $\{p_\lambda(x_1,\dots,x_n)\}_{\lambda\in\Par_n}$ can be expressed in terms of the mixed moments $\sigma_{n,k}$ and the dual mixed moments $\nu_{n,k}$ for $\{p_i(x)\}_{i\ge0}$ as follows:
\begin{equation}
  \label{eq:M_lambda,mu}
M_{\lambda,\mu}(n) = \det\left(\sigma_{\lambda_{i}+n-i,\mu_j+n-j}\right)_{i,j=1}^n,
\end{equation}
\begin{equation}
  \label{eq:N_lambda,mu}
N_{\lambda,\mu}(n) = \det\left(\nu_{\lambda_{i}+n-i,\mu_j+n-j}\right)_{i,j=1}^n. 
\end{equation}
In particular, $M_{\lambda,\mu}(n)=N_{\lambda,\mu}(n) = 0$ unless $\mu\subseteq\lambda$. 
\end{prop}
\begin{proof}
By Lemma~\ref{lem:pq} and the fact
\[
x^n = \sum_{k=0}^n \sigma_{n,k} p_k(x),
\]
we have
\[
\det\left( x_i^{\lambda_j+n-j}\right)_{i,j=1}^n
= \sum_{\mu\subseteq\lambda} \det\left(\sigma_{\lambda_{i}+n-i,\mu_j+n-j}\right)_{i,j=1}^n 
\det\left( p_{\mu_j+n-j}(x_i)\right)_{i,j=1}^n.
\]
Dividing both sides by $\Delta(x)$, we obtain \eqref{eq:M_lambda,mu}. 
By the same arguments, we obtain \eqref{eq:N_lambda,mu}.  
\end{proof}

\section{Lecture hall tableaux and lecture hall Schur functions}
\label{sec:lect-hall-tabl}

In this section we define lecture hall tableaux and their multivariate generating functions called  lecture hall Schur functions. We then study their combinatorial aspects.

\subsection{Lecture hall tableaux}

Recall that, for a cell $(i,j)$ in $\lambda$, the content $c(i,j)$ is defined by
$c(i,j)=j-i$. In the introduction we defined lecture hall tableaux of type
$(n,\ge,>)$ and $(n,<,\le)$. The definition can be extended to arbitrary types in
the obvious way.

\begin{defn}
Let $\prec_1$ and $\prec_2$ be inequalities in $\{<,\le,>,\ge\}$. For an integer $n$ and partitions $\mu\subseteq\lambda$ with $\ell(\lambda)\le n$, a \emph{lecture hall tableau of shape $\lm$ of type $(n,\prec_1,\prec_2)$} is a filling $T$ of the cells in the Young diagram of $\lm$ with nonnegative integers satisfying the following conditions:
\[
\frac{T(i,j)}{n+c(i,j)} \prec_1 \frac{T(i,j+1)}{n+c(i,j+1)}, \qquad
\frac{T(i,j)}{n+c(i,j)} \prec_2 \frac{T(i+1,j)}{n+c(i+1,j)}.
\]
We denote by $\LHT_{(n,\prec_1,\prec_2)}(\lm)$ the set of such fillings. 
\end{defn}

See Figure~\ref{fig:LHT} for an example of a lecture hall tableau of type $(n,\ge,>)$. 

For $T\in \LHT_{(n,\prec_1,\prec_2)}(\lm)$, the \emph{weight} $\wt(T)$ is defined by
\[
\wt(T)=\prod_{s\in \lm} y_{T(s)} u^{\flr{T(s)/(n+c(s))}} v^{o(\flr{T(s)/(n+c(s))})},
\]
where $o(m)$ is $1$ if $m$ is odd and $0$ otherwise. 

Let $\vec y=(y_0,y_1,\dots)$ be a sequence of variables. For a multivariate
function $f(\vec y)=f(y_0,y_1,\dots)$ we denote $f(\vec q)=f(1,q,q^2,\dots)$.
For a sequence $\alpha=(\alpha_1,\dots,\alpha_n)$ of nonnegative integers, we
denote $\vec y_\alpha = y_{\alpha_1}\cdots y_{\alpha_n}$.

\begin{defn}\label{defn:he}
The \emph{complete homogeneous lecture hall function} $h_k^{(n)} =h_k^{(n)}(\vec y;u,v)$ is defined by
\[
{h}_k^{(n)}(\vec y;u,v) = \sum_{\alpha} \vec y_{\alpha} u^{|\flr{\alpha}_S|} v^{o(\flr{\alpha}_S)},
\]
where $S=(n,n+1,\dots,n+k-1)$ and the sum is over all sequences $\alpha=(\alpha_1,\dots,\alpha_k)$ of integers satisfying
\[
\frac{\alpha_1}{n} \ge \frac{\alpha_2}{n+1} \ge \dots \ge \frac{\alpha_k}{n+k-1}\ge0.
\]
For $0\le k\le n$, the \emph{elementary lecture hall function} $e_k^{(n)} =e_k^{(n)}(\vec y;u,v)$ is defined by
\[
{e}_k^{(n)}(\vec y;u,v) =\sum_{\lambda} \vec y_{\lambda} u^{|\flr{\lambda}_S|} v^{o(\flr{\lambda}_S)},
\]
where $S=(n,n-1,\dots,n-k+1)$ and the sum is over all sequences $\lambda=(\lambda_1,\dots,\lambda_k)$ of integers satisfying
\[
\frac{\lambda_1}{n}>\frac{\lambda_2}{n-1}> \dots> \frac{\lambda_k}{n-k+1}\ge0.
\]
If $k>n$, we define $e_k^{(n)}=0$. 
\end{defn}

Note that
\[
h_k^{(n)}(\vec y;u,v) = \sum_{\alpha\in AL_{n+k-1,k}} \vec y_\alpha u^{|\flr{\alpha}|} v^{o(\flr{\alpha})}, \qquad
e_k^{(n)}(\vec y;u,v) =  \sum_{\lambda\in L_{n,k}} \vec y_\lambda u^{|\flr{\lambda}|} v^{o(\flr{\lambda})}.
\]
Recall the complete homogeneous symmetric functions 
\[
h_k(\vec y) = \sum_{i_1\ge\dots\ge i_k\ge0} y_{i_1} \cdots y_{i_k}
\]
 and the elementary symmetric functions 
\[
e_k(\vec y) = \sum_{i_1>\dots>i_k\ge0} y_{i_1} \cdots y_{i_k}.
\]
We have the following connections between these objects.

\begin{prop}\label{prop:hehe}
We have
\begin{align}
\label{eq:lim_h}
\lim_{n\to\infty} h_k^{(n)}(\vec y;u,v) &= h_k(\vec y), \\
\label{eq:lim_e}
\lim_{n\to\infty} e_k^{(n)}(\vec y;u,v) &= e_k(\vec y),\\
\label{eq:h00} 
h_k^{(n)}(\vec y;0,0) &= h_k(y_0,y_1,\dots,y_{n-1}), \\
\label{eq:e00} 
e_k^{(n)}(\vec y;0,0) &= e_k(y_0,y_1,\dots,y_{n-1}).
\end{align}
\end{prop}
\begin{proof}
We will only prove \eqref{eq:lim_h} and \eqref{eq:h00} because \eqref{eq:lim_e} and \eqref{eq:e00} can be proved similarly. 

First, we claim that for integers $i,j,N$ with $0\le i<N$, we have
$\frac{i}{N}\ge \frac{j}{N+1}$ if and only if $i\ge j$. The ``if part'' of the claim is clear. For the ``only if part'', suppose that
$\frac{i}{N}\ge \frac{j}{N+1}$ and $i<j$. Then $\frac{N+1}{N}\ge\frac{j}{i}\ge\frac{i+1}{i}$, which is a contradiction to the assumption
 $i<N$. 

We now prove \eqref{eq:lim_h}. By definition, 
\[
h_k^{(n)}(\vec y;u,v) = \sum_{\alpha} \vec y_{\alpha} u^{|\flr{\alpha}|} v^{o(\flr{\alpha})},
\]
where the sum is over all compositions $\alpha=(\alpha_1,\dots,\alpha_k)$ satisfying
\begin{equation}
  \label{eq:1}
\frac{\alpha_1}{n}\ge \frac{\alpha_2}{n+1}\ge \dots \ge \frac{\alpha_k}{n+k-1}\ge 0.  
\end{equation}
By the above claim, as $n\to\infty$, the condition \eqref{eq:1} is equivalent to
$\alpha_1\ge \dots\ge\alpha_k\ge0$, which  implies \eqref{eq:lim_h}. 

To prove \eqref{eq:h00}, observe that 
since $|\flr{\alpha}|=\flr{\alpha_1/n}+\cdots+\flr{\alpha_k/(n+k-1)}$, we have
\[
h_k^{(n)}(\vec y;0,0) = \sum_{\alpha} \vec y_{\alpha},
\]
where the sum is over all compositions $\alpha=(\alpha_1,\dots,\alpha_k)$ satisfying
\begin{equation}
\label{eq:8}
1>\frac{\alpha_1}{n}\ge \frac{\alpha_2}{n+1}\ge \dots \ge \frac{\alpha_k}{n+k-1}\ge 0.  
\end{equation}
By the claim again, the condition \eqref{eq:8} is equivalent to 
$n>\alpha_1\ge \dots\ge\alpha_k\ge0$, which  implies \eqref{eq:h00}. 
\end{proof}

By definition, we have
$h_k^{(n)}(\vec q;u,v) = \AL_{n+k-1,k}(u,v,q)$ and 
$e_k^{(n)}(\vec q;u,v) = L_{n,k}(u,v,q)$. 
Thus, we can rewrite Proposition~\ref{prop:mu,nu} as follows. 

\begin{prop}\label{prop:mu,h}
We have
\begin{align*}
\sigma^L_{n,k}(-uv,-u/v;q) &= \AL_{n,n-k}(u,v,q) = h_{n-k}^{(k+1)}(\vec q;u,v),\\
\nu^L_{n,k}(-uv,-u/v;q) &=(-1)^{n-k} L_{n,n-k}(u,v,q)
=(-1)^{n-k}  e_{n-k}^{(n)}(\vec q;u,v).
\end{align*}
\end{prop}

By \eqref{eq:inverse} and Proposition~\ref{prop:mu,h}, 
\[
\sum_{i=0}^m h^{(i+1)}_{m-i}(\vec q;u,v) (-1)^{i-n} e^{(i)}_{i-n}(\vec q;u,v)= 
\sum_{i=0}^m (-1)^{m-i} e^{(m)}_{m-i}(\vec q;u,v) h^{(n+1)}_{i-n}(\vec q;u,v) =\delta_{m,n}.
\]  
The following proposition is a multivariate analog of the above equations.

\begin{prop}\label{prop:eh}
We have
\begin{align}
  \label{eq:he}
\sum_{i=0}^m h^{(i+1)}_{m-i}(\vec y;u,v) (-1)^{i-n} e^{(i)}_{i-n}(\vec y;u,v) &=   \delta_{m,n},  \\
  \label{eq:eh}
\sum_{i=0}^m (-1)^{m-i} e^{(m)}_{m-i}(\vec y;u,v) h^{(n+1)}_{i-n}(\vec y;u,v)  &=\delta_{m,n}.  
\end{align}
\end{prop}
\begin{proof}
Let $f(m,n)$ be the left hand side of \eqref{eq:he}. 
If $m<n$, we have $f(m,n)=0$. If $m\ge n$, 
\[
f(m,n)=\sum_{i=0}^m \sum_{\alpha,\lambda} (-1)^{i-n}
\vec y_\alpha \vec y_\lambda u^{|\flr{\alpha}|+|\flr{\lambda}|}v^{o(\flr{\alpha})+o(\flr{\lambda})},
\]
where the second sum is over all sequences $\alpha$ and $\lambda$  of integers such that
\[
\frac{\alpha_1}{i+1}\ge \frac{\alpha_2}{i+2}\ge\cdots\ge\frac{\alpha_{m-i}}{m}\ge0,
\]
\[
\frac{\lambda_1}{i}>\frac{\lambda_2}{i-1}>\cdots>\frac{\lambda_{i-n}}{n+1}\ge0.
\]
Moving the first element of $\lambda$ to $\alpha$ or vice versa gives a
sign-reversing involution, which shows that $f(m,n)=\delta_{m,n}$. More
precisely, the sign of $(\alpha,\lambda)$ is $(-1)^{\ell(\alpha)}$ and the image
of $(\alpha,\lambda)$ under the sign-reversing involution is
\[
((\alpha_2,\ldots
,\alpha_{m-i}),(\alpha_1,\lambda_1,\ldots ,\lambda_{i-n}))
\]
if
$\frac{\alpha_1}{i+1}>\frac{\lambda_1}{i}$ and
\[
((\lambda_1, \alpha_1,\ldots
,\alpha_{m-i}),(\lambda_2,\ldots ,\lambda_{i-n}))
\]
 otherwise. The unique fixed
point is $(\emptyset,\emptyset)$.

The second identity \eqref{eq:eh} can be proved similarly.
\end{proof}

Replacing $n$ by $m-n$, we can rewrite \eqref{eq:he} as
\[
\sum_{i=0}^n (-1)^i e^{(m)}_i(\vec y;u,v) h^{(m-n+1)}_{n-i}(\vec y;u,v) = \delta_{n,0}.
\]  
If $m\to\infty$, the above equation becomes the well known identity
\[
\sum_{i=0}^n (-1)^i e_i(\vec y) h_{n-i}(\vec y) = \delta_{n,0}.
\]  

\subsection{Lecture hall Schur functions}

For partitions $\mu\subseteq\lambda$ with $\ell(\lambda)\le n$ and inequalities
$\prec_1$ and $\prec_2$, we define the \emph{lecture hall Schur function}
$LS_\lm^{(n,\prec_1,\prec_2)}(\vec y;u,v)$ by
\[
LS_\lm^{(n,\prec_1,\prec_2)}(\vec y;u,v) = \sum_{T\in \LHT_{(n,\prec_1,\prec_2)}(\lm)} \wt(T).
\]
By definition, we have
\[
h_k^{(n)}(\vec y;u,v) = LS^{(n,\ge,>)}_{(k)}(\vec y;u,v), \qquad e_k^{(n)}(\vec y;u,v) = LS^{(n,\ge,>)}_{(1^k)}(\vec y;u,v).
\]
Using the same arguments as in the proof of Proposition~\ref{prop:hehe}, we can prove the following proposition. 

\begin{prop}\label{prop:LS_S}
Suppose that $\succ_1$ and $\succ_2$ are any inequalities in $\{>,\ge\}$, and $\prec_1$ and $\prec_2$ are any inequalities in $\{<,\le\}$. 
Then for partitions $\lambda$ and $\mu$ with $\mu\subseteq\lambda$ and $\ell(\lambda)\le n$, 
\begin{align}
\notag
\lim_{n\to\infty} LS_\lm^{(n,\succ_1,\succ_2)}(\vec y;u,v) &= s_{\lm}(\vec y),\\
\notag
\lim_{n\to\infty} LS_\lm^{(n,\prec_1,\prec_2)}(\vec y;u,v) &= s_{\lambda'/\mu'}(\vec y),\\
\label{eq:LS00}
LS_\lambda^{(n,\succ_1,\succ_2)}(\vec y;0,0) &= s_{\lambda}(y_0,y_1,\dots,y_{n-1}).
\end{align}
\end{prop}

\begin{remark}
Note that \eqref{eq:LS00} is not true if we use a skew Young diagram $\lm$ or the other inequalities $\prec_1$ and $\prec_2$. For every lecture hall tableau contributing a nonzero weight in $LS_\lambda^{(n,\succ_1,\succ_2)}(\vec y;0,0)$, the entry in the cell $(1,1)$ is the largest. Thus every entry is at most $n-1$. On the contrary $LS_\lm^{(n,\succ_1,\succ_2)}(\vec y;0,0)$
and $LS_\lambda^{(n,\prec_1,\prec_2)}(\vec y;0,0)$ do not have such a property.
\end{remark}

There are Jacobi--Trudi type formulas for $LS_\lm^{(n,\ge,>)}(\vec y;u,v)$ and $LS_\lm^{(n,<,\le)}(\vec y;u,v)$. 

\begin{thm}\label{thm:schur}
Let $\lambda$ and $\mu$ be partitions with $\ell(\lambda)\le n$ and $\mu\subseteq\lambda$.  Then
\begin{align}
LS_\lm^{(n,\ge,>)}(\vec y;u,v)  
&= \det\left(h_{\lambda_i-\mu_j-i+j}^{(n-j+1+\mu_j)}(\vec y;u,v)\right)_{i,j=1}^{\ell(\lambda)}, \label{eq:JT1}\\
LS_\lm^{(n,\ge,>)}(\vec y;u,v)  
&= \det\left(e_{\lambda'_i-\mu'_j-i+j}^{(n+j-1-\mu_j')}(\vec y;u,v)\right)_{i,j=1}^{\ell(\lambda')}. \label{eq:JT2}
\end{align}
\end{thm}

\begin{thm}\label{thm:schur2}
Let $\lambda$ and $\mu$ be partitions with $\ell(\lambda)\le n$ and $\mu\subseteq\lambda$.  Then
\begin{align*}
LS_\lm^{(n,<,\le)}(\vec y;u,v)  
&= \det\left(e_{\lambda_i-\mu_j-i+j}^{(n-i+\lambda_i)}(\vec y;u,v)\right)_{i,j=1}^{\ell(\lambda)}, \\
LS_\lm^{(n,<,\le)}(\vec y;u,v)  
&= \det\left( h_{\lambda'_i-\mu'_j-i+j}^{(n+i-\lambda'_i)}(\vec y;u,v)\right)_{i,j=1}^{\ell(\lambda')}. 
\end{align*}
\end{thm}

Note that if $n\to\infty$ in Theorems~\ref{thm:schur} and \ref{thm:schur2}, we
obtain the usual Jacobi--Trudi formula for $s_{\lm}(\vec y)$.
Theorems~\ref{thm:schur} and \ref{thm:schur2} show that lecture hall Schur
functions are a special case of Macdonald's 9th variation of Schur functions
\cite{Macdonald_Schur}. Nakagawa et at.~\cite{NNSY} found a combinatorial
interpretation for Macdonald's 9th variation of Schur functions in terms of
tableaux. It will be interesting to find a connection between our result and
theirs.

In the next
subsection we prove Theorem~\ref{thm:schur}. The proof of
Theorem~\ref{thm:schur2} is similar and thus omitted.

\subsection{Lecture hall lattice paths}

We define the \emph{lecture hall lattice} to be the infinite graph $G=(V,E)$,
where 
\[
V=\bigcup_{i\ge1, j\ge0} \left\{\left (i-1,\frac ji\right) , \left(i,\frac ji\right) \right\}
\]
and two distinct vertices $(a,b),(c,d)\in V$ with $a\le c$ are adjacent if one of the following conditions holds:
\begin{itemize}
\item $a=c-1$ and $b=d$,
\item $a=c$ and there is no vertex $(a,e)\in V$ with $b<e<d$ or $d<e<b$.
\end{itemize}
We define the \emph{strict lecture hall lattice} to be the infinite graph $G_s=(V_s,E_s)$,
where 
\[
V_s=\bigcup_{i\ge1, j\ge0} \left\{\left (i-1,\frac ji-\frac1{i^2}\right) , \left(i,\frac ji\right) \right\}
\]
and two distinct vertices $(a,b),(c,d)\in V_s$ with $a\le c$ are adjacent if one of the following conditions holds:
\begin{itemize}
\item $a=c-1$ and $b=d-1/c^2$,
\item $a=c$ and there is no vertex $(a,e)$ with $b<e<d$ or $d<e<b$.
\end{itemize}

A \emph{north step} is a pair $(u,v)$ of adjacent vertices in $G$ or $G_s$ of the form $u=(a,b)$
and $v=(a,c)$ with $b<c$. An \emph{east step} (resp.~\emph{west step}) is a pair $(u,v)$ of adjacent vertices in $G$ of the form $u=(a,b)$ and $v=(a+1,b)$ (resp.~$v=(a-1,b)$). A \emph{northeast step} (resp.~\emph{southwest step}) is a pair $(u,v)$ of adjacent vertices in $G_s$ of the form $u=(a,b)$ and $v=(a+1,b+\frac{1}{(a+1)^2})$ (resp.~$v=(a-1,b-\frac{1}{a^2})$).

An \emph{NW-path} (resp.~\emph{NE-path}) from $A$ to $B$ is a sequence $(u_1,u_2,\dots,u_k)$ of vertices in $V$ (resp.~$V_s$) such that
$u_1=A$, $u_k=B$ and $(u_i,u_{i+1})$ is a north step or a west step (resp.~a north step, or a northeast step) 
for $1\le i\le k-1$. For vertices $A$ and $B$ in  $V$ (resp.~$V_s$), we denote by $\NW(A,B)$ (resp.~$\NE(A,B)$) the set of NW-paths (resp.~NE-paths) from $A$ to $B$. If $B=(b,\infty)$, we define $\NW(A,B)$ (resp.~$\NE(A,B)$) to be the set of infinite sequences 
$(u_1,u_2,\dots)$, also called NW-paths (resp.~NE-paths), such that $u_1=A$, $\lim_{k\to\infty} u_k= B$ and $(u_1,\dots,u_k)$ is an NW-path (resp.~NE-path) for each $k$. 

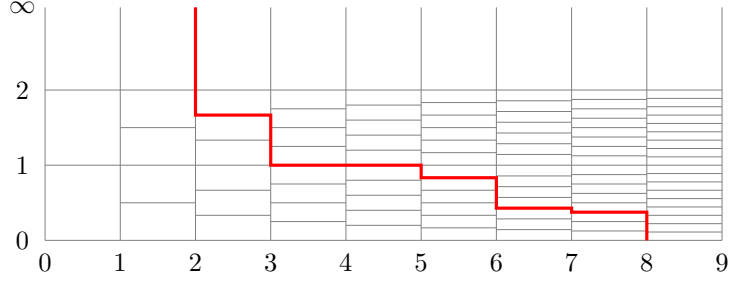
\begin{figure}
  \centering
\begin{tikzpicture}
\LHL{9}2
\draw [red,very thick] (2,3.1) -- (2,5/3) -- (3,5/3) -- (3,1) -- (5,1) -- (5,5/6) -- (6,5/6)
-- (6,3/7) -- (7,3/7) -- (7,3/8) -- (8,3/8) -- (8,0);
\end{tikzpicture}
\caption{A path in $\NW((8,0),(2,\infty))$. This path corresponds to the anti-lecture hall composition 
$\alpha=(\alpha_1,\dots,\alpha_6)=(5,4,5,5,3,3)$, which satisfy
$\frac{\alpha_1}{3}\ge \frac{\alpha_2}{4}\ge \frac{\alpha_3}{5}\ge \frac{\alpha_4}{6}\ge \frac{\alpha_5}{7}\ge \frac{\alpha_6}{8}\ge0$.}
\label{fig:NW}
\end{figure}

For $0\le k\le n$, consider an NW-path $p\in\NW((n,0), (k,\infty))$. For $1\le i\le n-k$, let $w_i = ((a_i,b_i),(a_i-1,b_i))$ be the $i$th leftmost west step in $p$. 
Define $\alpha=(\alpha_1,\dots,\alpha_{n-k})$ to be the composition given by $\alpha_i = a_ib_i$. Note that $\alpha_i$ can be considered as the number of regions in the lecture hall lattice below the step $w_i$.
It is easy to see that the map $p\mapsto \alpha$ is a bijection from $\NW((n,0), (k,\infty))$ to the set $\AL_{n,n-k}$ of anti-lecture hall compositions
$\alpha=(\alpha_1,\dots,\alpha_{n-k})$ satisfying
\[
\frac{\alpha_1}{k+1}\ge \frac{\alpha_2}{k+2}\ge \cdots \ge \frac{\alpha_{n-k}}{n}\ge0.
\]
See Figure~\ref{fig:NW} for an example of this correspondence.

For $p\in \NW((n,0), (k,\infty))$, we define 
\[
\wt(p) = \prod_{w=((i,j),(i-1,j))}y_{ij} u^{\flr{j}} v^{o(\flr{j})},
\]
where the product is over all west steps $w$ in $p$. Note that if $\alpha$ is the anti-lecture hall composition corresponding to $p$, we have
$\wt(p) = y_\alpha u^{\flr{\alpha}}v^{o(\flr{\alpha})}$. We also define
\[
W_{i,j} = \sum_{p\in \NW((i,0),(j,\infty))} \wt(p).
\]

By the correspondence between $\NW((i,0), (j,\infty))$ and $\AL_{i,i-j}$, we have
\begin{equation}
  \label{eq:W}
W_{i,j} =h^{(j+1)}_{i-j}(\vec y;u,v).  
\end{equation}

Now consider an NE-path
$p\in\NE((k,-1/(k+1)^2), (n,\infty))$. For $1\le i\le n-k$, let $e_i = ((a_i-1,b'_i),(a_i,b_i))$ be the $i$th rightmost step among all northeast steps in $p$. Define $\lambda=(\lambda_1,\dots,\lambda_{n-k})$ to be the partition given by $\lambda_i = a_ib_i$. Note that $\lambda_i$ can be considered as the number of regions in the strict lecture hall lattice below the step $e_i$. Similarly to the NW-path case, one can check that the map $p\mapsto \lambda$ is a bijection from $\NE((k, -1/(k+1)^2), (n,\infty))$ to the set $L_{n,n-k}$ of lecture hall partitions
$\lambda=(\lambda_1,\dots,\lambda_{n-k})$ satisfying
\[
\frac{\lambda_1}{n}> \frac{\lambda_2}{n-1}> \cdots > \frac{\lambda_{n-k}}{k+1}\ge0.
\]
See Figure~\ref{fig:NE} for an example of this correspondence.

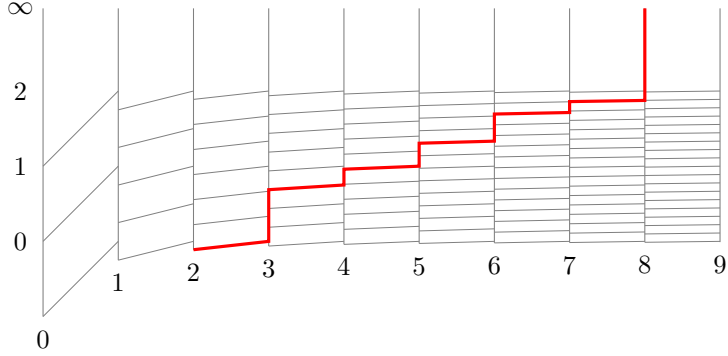
\begin{figure}
  \centering
\begin{tikzpicture}
\SLHL{9}2
\draw [red, very thick] (2,-1/3^2)--(3,0)--(3,3/4-1/16)--(4,3/4)--(4,1-1/25)--(5,5/5)--(5,8/6-1/36)--(6,8/6)--(6,12/7-1/49)--(7,12/7)--(7,15/8-1/64)--(8,15/8)--(8,3.1);
\end{tikzpicture}
\caption{A path in $\NE((2,-1/3^2),(8,\infty))$. This path corresponds to the lecture hall partition
$\lambda=(\lambda_1,\dots,\lambda_6)=(15,12,8,5,3,0)$ satisfying
$\frac{\lambda_1}{8}> \frac{\lambda_2}{7}> \frac{\lambda_3}{6}> \frac{\lambda_4}{5}> \frac{\lambda_5}{4}> \frac{\lambda_6}{3}\ge0$.}
\label{fig:NE}
\end{figure}

For $p\in \NE((k, -1/(k+1)^2), (n,\infty))$, we define 
\[
\wt(p) = \prod_{e=((i-1,j'),(i,j))}y_{ij} u^{\flr{j}} v^{o(\flr{j})},
\]
where the product is over all east or northeast steps $e$ in $p$. Note that if $\lambda$ is the lecture hall partition corresponding to $p$, we have
$\wt(p) = \vec y_\lambda u^{\flr{\lambda}}v^{o(\flr{\lambda})}$. We also define
\[
E_{i,j} = \sum_{p\in \NE((i,-1/(i+1)^2),(j,\infty))} \wt(p).
\]

By the correspondence between $\NE((i,-1/(i+1)^2), (j,\infty))$ and $L_{j,j-i}$, we have
\begin{equation}
  \label{eq:E}
E_{i,j} =e^{(j)}_{j-i}(\vec y;u,v).  
\end{equation}

\begin{lem}\label{lem:NILP}
Let $\mu=(\mu_1,\dots,\mu_n)\subseteq\lambda=(\lambda_1,\dots,\lambda_n)$.
For $1\le i\le n$, let $A_i=(\lambda_i+n-i,0)$ and $B_i=(\mu_{i}+n-i,\infty)$. Then
\[
LS_\lm^{(n,\ge,>)}(\vec y;u,v) = \sum_{(p_1,\dots,p_n)\in W} \wt(p_1)\cdots \wt(p_n),
\]  
where $W$ is the set of $n$-tuples $(p_1,\dots,p_n)$ of non-intersecting NW-paths with $p_i\in\NW(A_i,B_i)$ for $1\le i\le n$. 
\end{lem}
\begin{proof}
We construct a bijection $\phi:\LHT_{(n,\ge,>)}(\lm)\to W$ as follows. Let
$T\in\LHT_{(n,\ge,>)}(\lm)$. The $i$th row of $T$ satisfies 
\[
\frac{T_{i,\mu_i+1}}{\mu_i+n-i+1}\ge \frac{T_{i,\mu_i+2}}{\mu_i+n-i+2}\ge
\cdots\ge \frac{T_{i,\lambda_i}}{\lambda_i+n-i}.
\]

\begin{figure}
  \centering
\begin{tikzpicture}
\LHL{11}2
\draw [red, very thick] (7,3.1)--(7,9/8)--(8,9/8)--(8,4/9)--(9,4/9)--(9,3/10)--(10,3/10)--(10,0);
\draw [red, very thick] (4,3.1)--(4,5/5)--(5,5/5)--(6,6/6)--(6,4/7)--(7,4/7)--(7,3/8)--(8,3/8)--(8,1/9)--(9,1/9)--(9,0);
\draw [red, very thick] (2,3.1)--(2,2/3)--(3,2/3)--(3,2/4)--(4,2/4)--(4,1/5)--(5,1/5)--(5,0)--(6,0);
\draw [red, very thick] (1,3.1)--(1,1/2)--(2,1/2)--(2,0)--(4,0);
\end{tikzpicture}
\caption{The non-intersecting NW-paths corresponding to the lecture hall tableau in Figure~\ref{fig:LHT}.}
\label{fig:phi}
\end{figure}
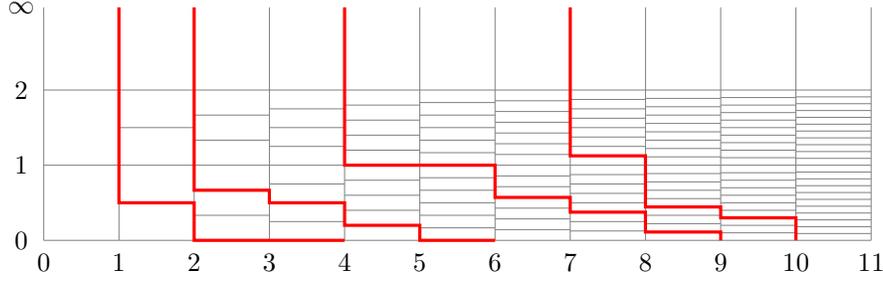

By the same arguments deriving \eqref{eq:W}, the $i$th row $T$ corresponds to the path $p_i\in\NW(A_i,B_i)$ whose $k$th leftmost west step is $((r,s),(r-1,s))$, where $r=\mu_i+n-i+k$ and $s=T_{i,\mu_i+k}/(\mu_i+n-i+k)$. 
Then we define $\phi(T)=(p_1,\dots,p_n)$, see Figure~\ref{fig:phi}. Since $T\in\LHT_{(n,\ge,>)}(\lm)$, we have 
$(p_1,\dots,p_n)\in N$. It is easy to see that the map $\phi$ is a bijection and
$\wt(T)=\wt(p_1)\cdots\wt(p_n)$. This completes the proof. 
\end{proof}

The following lemma is a dual version of Lemma~\ref{lem:NILP}. 

\begin{lem}\label{lem:NILP'}
Let $\mu=(\mu_1,\dots,\mu_n)\subseteq\lambda=(\lambda_1,\dots,\lambda_n)$.
For $1\le i\le \ell(\lambda')$, let $A_i=(n-\lambda'_i+i-1,0)$ and $B_i=(n-\mu'_i+i-1,\infty)$. Then
\[
LS_\lm^{(n,\ge,>)}(\vec y;u,v) = \sum_{(p_1,\dots,p_{\ell(\lambda')})\in E} \wt(p_1)\cdots \wt(p_{\ell(\lambda')}),
\]  
where $E$ is the set of $n$-tuples $(p_1,\dots,p_{\ell(\lambda')})$ of non-intersecting NE-paths with $p_i\in\NE(A_i,B_i)$ for $1\le i\le \ell(\lambda')$. 
\end{lem}
\begin{proof}
This can be proved by the same arguments as in the proof of Lemma~\ref{lem:NILP} except that we make the NE-path $p_i$ from the entries of the $i$th column of $T\in \LHT_{(n\ge,>)}(\lm)$, see Figure~\ref{fig:phi2}. We omit the details.  
\end{proof}

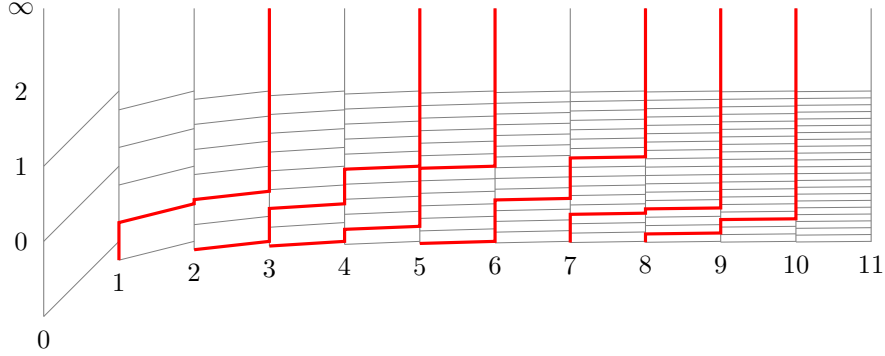
\begin{figure}
  \centering
\begin{tikzpicture}
\SLHL{11}2
\draw [red, very thick] (8,-1/81)--(8,1/9-1/81)--(9,1/9)--(9,3/10-1/100)--(10,3/10)--(10,3.1);
\draw [red, very thick] (7,-1/64)--(7,3/8-1/64)--(8,3/8)--(8,4/9-1/81)--(9,4/9)--(9,3.1);
\draw [red, very thick] (5,-1/36)--(6,0)--(6,4/7-1/49)--(7,4/7)--(7,9/8-1/64)--(8,9/8)--(8,3.1);
\draw [red, very thick] (3,-1/16)--(4,0)--(4,1/5-1/25)--(5,1/5)--(5,1-1/36)--(6,6/6)--(6,3.1);
\draw [red, very thick] (2,-1/9)--(3,0)--(3,2/4-1/16)--(4,2/4)--(4,1-1/25)--(5,5/5)--(5,3.1);
\draw [red, very thick] (1,-1/4)--(1,1/2-1/4)--(2,1/2)--(2,2/3-1/9)--(3,2/3)--(3,3.1);
\end{tikzpicture}
\caption{The non-intersecting NE-paths corresponding to the lecture hall tableau in Figure~\ref{fig:LHT}.}
\label{fig:phi2}
\end{figure}

\begin{proof}[Proof of  Theorem~\ref{thm:schur}]
By Lemma~\ref{lem:NILP} and the Lindstr\"om--Gessel--Viennot lemma, we have
\[
LS_\lm^{(n,\ge,>)}(\vec y;u,v) = \det(W_{\lambda_i+n-i,\mu_j+n-j})_{i,j=1}^n.
\]  
By \eqref{eq:W}, we obtain \eqref{eq:JT1}. Similarly, by Lemma~\ref{lem:NILP'}, we have
\[
LS_\lm^{(n,\ge,>)}(\vec y;u,v) = \det(E_{n-\lambda'_i+i-1,n-\mu'_j+j-1})_{i,j=1}^{\ell(\lambda')}.
\]
By \eqref{eq:E}, we obtain \eqref{eq:JT2}.
\end{proof}

\section{Multivariate little $q$-Jacobi polynomials}
\label{sec:mult-little-q}

In this section we define multivariate little $q$-Jacobi polynomials and find a combinatorial interpretation for their mixed and dual mixed moments. 

For $0<q<1$, the \emph{$q$-integral} is defined by
\[
\int_0^a f(x) d_qx = (1-q) \sum_{n=0}^\infty f(aq^n) aq^n,
\]
\[
\int_a^b f(x) d_qx = \int_0^b f(x) d_qx -\int_0^a f(x) d_qx.
\]

Let $x=(x_1,\dots,x_n)$.
For $f,g\in\CC[x_1,\dots,x_n]$, $0<a<1/q$, $b<1/q$ and $0<t<1$, we define
\[
\langle f,g\rangle_{L,t}^{a,b} =
\int_{x_1=0}^1\int_{x_2=0}^{tx_1}\cdots\int_{x_n=0}^{tx_{n-1}}
f(x) g(x) v(x;a,b,t)d_qx,
\]
where
\[
v(x;a,b,t)=\Delta(x)\prod_{i=1}^n x_i^{\alpha} \frac{(qx_i)_\infty}{(qbx_i)_\infty}
\prod_{1\le i<j\le n} x_i^{2\tau-1}(q^{1-\tau}x_j/x_i)_{2\tau-1},
\]
$a=q^\alpha$ and $t=q^\tau$.

The \emph{multivariate little $q$-Jacobi polynomials} $\{p_\lambda^L(x;a,b;q,t):\lambda\in\Par_n\}$ are defined by the following conditions:
  \begin{enumerate}
  \item $p_\lambda^L(t) = m_\lambda+\sum_{\mu<\lambda} d_{\lambda,\mu}(t)m_\mu$, where $<$ is the dominance order and 
$m_\lambda$ is the monomial symmetric polynomial, 
  \item $\langle p_\lambda^L(t),m_\mu\rangle_{L,t}^{a,b}=0$ if $\mu<\lambda$.
  \end{enumerate}

It is known \cite[Theorem~5.1]{StokmanKoornwinder97} that
$p_\lambda^L(x;a,b;q,t)$ is obtained as a limit transition from the Koornwinder polynomial \cite{Koornwinder92}, which  is the $BC_n$-type Macdonald polynomial generalizing the Askey--Wilson polynomial. 
 
In this paper we consider the case $t=q$ of the multivariate little $q$-Jacobi polynomials, i.e.,
\[
p^L_\lambda(x_1,\dots,x_n;a,b;q)= p_\lambda^L(x_1,\dots,x_n;a,b;q,q).
\]
In this case $p^L_\lambda(x_1,\dots,x_n;a,b;q)$ is orthogonal with respect to the linear functional 
$\LLL^{L,q}_{a,b}$ given the following formula, see \cite[(5.19)]{Stokman97}:
\begin{equation}
  \label{eq:LL}
\LLL^{L,q}_{a,b}(f(x))=  \int_{[0,1]^n} f(x)  \Delta(x)^2
\prod_{i=1}^n x_i^{\alpha} \frac{(qx_i)_\infty}{(qbx_i)_\infty} d_qx, \qquad (a=q^\alpha).
\end{equation}

Stokman \cite[Proposition~5.9]{Stokman97} showed that if $t=q$, the multivariate little $q$-Jacobi polynomial can be written as a determinant of little $q$-Jacobi polynomials:
\[
p^L_\lambda(x_1,\dots,x_n;a,b;q) = 
\frac{\det\left( p^L_{\lambda_j+n-j}(x_i;a,b;q)\right)_{i,j=1}^n}{\Delta(x)}.
\]
Thus we can consider the mixed moment $M^{L}_{\lambda,\mu}(n;a,b;q)$ and the dual mixed moment $N^L_{\lambda,\mu}(n;a,b;q)$ of the multivariate little $q$-Jacobi polynomials $p^L_\lambda(x_1,\dots,x_n;a,b;q)$. They satisfy
\begin{align}
\notag
s_\lambda(x_1,\dots,x_n) &= \sum_{\mu\subseteq\lambda} M^L_{\lambda,\mu}(n;a,b;q) p^L_\mu(x_1,\dots,x_n;a,b;q), \\  
  \label{eq:pL->s}
p^L_\lambda(x_1,\dots,x_n;a,b;q) &= \sum_{\mu\subseteq \lambda} N^L_{\lambda,\mu}(n;a,b;q) s_\mu(x_1,\dots,x_n).  
\end{align}

The following theorem implies that
the mixed moments and the dual mixed moments are
generating functions for lecture hall tableaux.

\begin{thm}\label{thm:Jacobi->Schur}
For a partition $\lambda$ with at most $n$ parts
and a partition $\mu$ with $\mu\subseteq \lambda$, we have
\begin{align*}
N^L_{\lambda,\mu}(n;-uv,-u/v;q)&=(-1)^{|\lm|} LS^{(n,<,\le)}_{\lm}(\vec q;u,v),\\  
M^L_{\lambda,\mu}(n;-uv,-u/v;q)&=LS^{(n,\ge,>)}_{\lm}(\vec q;u,v).
\end{align*}
Equivalently,
\begin{align*}
p^L_\lambda(x_1,\dots,x_n;-uv,-u/v;q) 
&= \sum_{\mu\subseteq\lambda} (-1)^{|\lm|} LS^{(n,<,\le)}_{\lm}(\vec q;u,v)s_\mu(x_1,\dots,x_n),\\
s_\lambda(x_1,\dots,x_n) 
&= \sum_{\mu\subseteq\lambda} LS^{(n,\ge,>)}_{\lm}(\vec q;u,v)p^L_\mu(x_1,\dots,x_n;-uv,-u/v;q).
\end{align*}
\end{thm}
\begin{proof}
By \eqref{eq:pL->s}, in order to prove the first identity it suffices to show
\[
N^L_{\lambda,\mu}(n;-uv,-u/v;q)=(-1)^{|\lm|} LS^{(n,<,\le)}_{\lm}(\vec q;u,v).
\]
By \eqref{eq:N_lambda,mu}, Proposition~\ref{prop:mu,h} and
Theorem~\ref{thm:schur2},
\begin{align*}
N^L_{\lambda,\mu}(n;-uv,-u/v;q) &= \det\left(\nu_{\lambda_{i}+n-i,\mu_j+n-j}(-uv,-u/v;q)\right)_{i,j=1}^n\\
&= \det\left((-1)^{\lambda_i-\mu_j-i+j} e_{\lambda_i-\mu_j-i+j}^{(\lambda_i+n-i)} (\vec q,u,v)\right)_{i,j=1}^n\\
&= (-1)^{|\lm|} LS^{(n,<,\le)}_{\lm}(\vec q;u,v),
\end{align*}
which establishes the first identity. The second identity can be proved similarly.
\end{proof}

By taking the limit $n\to\infty$ in Theorem~\ref{thm:Jacobi->Schur}, we obtain an infinite variable symmetric function.
\begin{cor}
There is an infinite variable polynomial $p^L_\lambda(x_1,x_2,\dots;a,b;q)$ such that
\[
p^L_\lambda(x_1,x_2,\dots;a,b;q)  = \lim_{n\to\infty} p^L_\lambda(x_1,x_2,\dots,x_n;a,b;q) 
\]
and
\[
p^L_\lambda(x_1,x_2,\dots;a,b;q) 
= \sum_{\mu\subseteq\lambda} (-1)^{|\lm|} s_{\lambda'/\mu'}(1,q,q^2,\dots)s_\mu(x_1,x_2,\dots).
\]
\end{cor}

Note that $p^L_\lambda(x_1,x_2,\dots;a,b;q)$ is independent of $a$ and $b$. This
is a special case of a variation of Schur function $s_{\lm}(x/y)$ in
\cite[Exercise 23, p.90]{Macdonald} defined by
\[
s_\lambda(x/y)= \sum_{\mu\subseteq\lambda} (-1)^{|\lm|} 
s_\mu(x_1,x_2,\dots)s_{\lambda'/\mu'}(y_1,y_2,\dots).
\]

\section{Moments and dual moments of multivariate little $q$-Jacobi polynomials}
\label{sec:moments-dual-moments}

In this section we prove product formulas for moments and dual moments of
the multivariate little $q$-Jacobi polynomials. 

The \emph{moment} $M^L_\lambda(n;a,b;q)$ of the multivariate little $q$-Jacobi polynomials $p^L_\lambda(x_1,\dots,x_n;a,b;q)$ is defined by
\[
M^L_\lambda(n;a,b;q) =M^L_{\lambda,\emptyset}(n;a,b;q)=\frac{\LLL^{L,q}_{a,b}(s_{\lambda}(x_1,\dots,x_n))}{\LLL^{L,q}_{a,b}(1)},
\]
where $\LLL^{L,q}_{a,b}$ is given in \eqref{eq:LL}.

The following is a Selberg-type integral due to Kadell \cite{Kadell1988a}, see also \cite[Corollary~1.3]{Warnaar2005}:
\begin{multline}
  \label{eq:selberg}
\frac{[n]_q!}{n!} \int_{[0,1]^n} s_\lambda(x_1,\dots,x_n) \Delta(x_1,\dots,x_n)^2
\prod_{i=1}^n x_i^{\alpha-1} \frac{(qx_i;q)_\infty}{(q^\beta x_i;q)_\infty} d_qx_1\dots d_qx_n\\
=q^{\alpha\binom n2+2\binom n3} \prod_{1\le i<j\le n} \frac{q^{\lambda_j+n-j}-q^{\lambda_i+n-i}}{q^{i-1}-q^{j-1}} 
\prod_{i=1}^n \frac{\Gamma_q(\alpha+n-i+\lambda_i)\Gamma_q(\beta+i-1)\Gamma_q(i+1)}
{\Gamma_q(\alpha+\beta+2n-i-1+\lambda_i)},
\end{multline}
where $\Gamma_q(z)=(1-q)^{1-z}(q;q)_\infty/(q^z;q)_\infty$.
We refer the reader to \cite{Forrester2008} for more information on the Selberg integral. 

Using \eqref{eq:selberg} we obtain a product formula for 
the moment $M^L_\lambda(n;a,b;q)$. 

\begin{thm}\label{thm:ML}
For a partition $\lambda=(\lambda_1,\dots ,\lambda_n)$, we have
\[
M^L_\lambda(n;a,b;q) = 
\prod_{1\le i<j\le n} \frac{q^{\lambda_j+n-j}-q^{\lambda_i+n-i}}{q^{i-1}-q^{j-1}}
\prod_{i=1}^n \frac{(aq^{n-i+1})_{\lambda_i}}{(abq^{2n-i+1})_{\lambda_i}}.
\]
\end{thm}
\begin{proof}
Let $a=q^{\alpha}$ and $b=q^{\beta}$. By \eqref{eq:LL} and \eqref{eq:selberg}, 
\begin{align*}
M^L_\lambda(n;a,b;q) &=\frac{\LLL^{L,q}_{a,b}(s_{\lambda}(x_1,\dots,x_n))}{\LLL^{L,q}_{a,b}(s_{\emptyset}(x_1,\dots,x_n))}\\
&= \prod_{1\le i<j\le n} \frac{q^{\lambda_j+n-j}-q^{\lambda_i+n-i}}{q^{i-1}-q^{j-1}}  
\prod_{i=1}^n \frac{\Gamma_q(\alpha+n-i+1+\lambda_i) \Gamma_q(\alpha+\beta+2n-i+1)}
{\Gamma_q(\alpha+1+n-i)\Gamma_q(\alpha+\beta+2n-i+1+\lambda_i)}.
\end{align*}
This is the same as the desired identity. 
\end{proof}

We note that the connection between the Jacobi polynomials and the Selberg integral 
has been observed by Aomoto \cite{Aomoto1987} and further studied by many people, see for example \cite{Luque2003,Mimachi1998, Olshanski_Selberg1, Olshanski_Selberg2, Stokman97}.
 
By Theorems~\ref{thm:Jacobi->Schur} and \ref{thm:ML}, we obtain a product formula
for $LS^{(n,\ge,>)}_{\lambda}(\vec q;u,v)$.

\begin{cor}\label{cor:LHT_prod}
For a partition $\lambda=(\lambda_1,\dots ,\lambda_n)$, we have
\[
LS^{(n,\ge,>)}_{\lambda}(\vec q;u,v)=\prod_{1\le i<j\le n} \frac{q^{\lambda_j+n-j}-q^{\lambda_i+n-i}}{q^{i-1}-q^{j-1}}
\prod_{i=1}^n \frac{(-uvq^{n-i+1})_{\lambda_i}}{(u^2q^{2n-i+1})_{\lambda_i}}.
\]  
\end{cor}

Note that by Proposition~\ref{prop:LS_S}, if we set $u=v=0$ in Corollary~\ref{cor:LHT_prod}, we obtain the well known identity for the principal specialization of the Schur function:
\[
s_{\lambda}(1,q,\dots,q^{n-1}) = \prod_{1\le i<j\le n} \frac{q^{\lambda_j+n-j}-q^{\lambda_i+n-i}}{q^{i-1}-q^{j-1}}. 
\]

We now consider the \emph{dual moment} $N^L_\lambda(n;a,b;q)$ defined by
\[
N^L_\lambda(n;a,b;q) = N^L_{\lambda,\emptyset}(n;a,b;q).
\]
Since
\[
p^L_\lambda(x_1,\dots,x_n;a,b;q) = \sum_{\mu\subseteq \lambda} N^L_{\lambda,\mu}(n;a,b;q) s_\mu(x_1,\dots,x_n),
\]
the dual moment $N^L_\lambda(n;a,b;q)$ is the constant term $p^L_\lambda(0,\dots,0;a,b;q)$ of $p^L_\lambda(x_1,\dots,x_n;a,b;q)$. 

We also have a product formula for the dual moment $N^L_\lambda(n;a,b;q)$.
\begin{thm}\label{thm:NL}
For a partition $\lambda=(\lambda_1,\dots ,\lambda_n)$, we have
\begin{multline*}
N^L_\lambda(n;a,b;q) = (-1)^{|\lm|} q^{n(\lambda')-n(\lambda)}
\prod_{1\le i<j\le n} \frac{q^{\lambda_j+n-j}-q^{\lambda_i+n-i}}{q^{i-1}-q^{j-1}}
\prod_{i=1}^n \frac{(aq^{n-i+1})_{\lambda_i}}{(abq^{n-i+1+\lambda_i})_{n-i+\lambda_i}}\\
\times \prod_{1\le i<j\le n}(1-abq^{2n+\lambda_i+\lambda_j-i-j+1}),
\end{multline*}
where $n(\lambda)=\sum_{i=1}^{\ell(\lambda)} (i-1)\lambda_i$.
\end{thm}

Before proving the above theorem, we present its corollary. By Theorems~\ref{thm:Jacobi->Schur} and \ref{thm:NL}, we obtain a product formula
for $LS^{(n,<,\le)}_{\lambda}(\vec q;u,v)$. 

\begin{cor}\label{cor:LHT_prod2}
For a partition $\lambda=(\lambda_1,\dots ,\lambda_n)$, we have
\begin{multline*}
LS^{(n,<,\le)}_{\lambda}(\vec q;u,v) =  q^{n(\lambda')-n(\lambda)} \prod_{1\le i<j\le n} \frac{q^{\lambda_j+n-j}-q^{\lambda_i+n-i}}{q^{i-1}-q^{j-1}} 
\prod_{i=1}^n \frac{(-uvq^{n-i+1})_{\lambda_i}}{(u^2q^{n-i+1+\lambda_i})_{n-i+\lambda_i}}\\
\times \prod_{1\le i<j\le n}(1-u^2q^{2n+\lambda_i+\lambda_j-i-j+1}).
\end{multline*}
\end{cor}

In order to prove Theorem~\ref{thm:NL}, we need the following lemma. 

\begin{lem}\label{lem:det}
We have
  \begin{equation}
    \label{eq:5}
\det\left(\frac{1}{(ax_j)_i(b/x_j)_i}\right)_{i,j=1}^n
=(-1)^nb^{-n^2}q^{-\binom{n+1}3} x_1\cdots x_n
\frac{\prod_{1\le i<j\le n} (b-aq^{n-1}x_ix_j)(x_j-x_i)}{\prod_{j=1}^n (ax_j)_n(q^{1-n}b^{-1}x_j)_n}.
  \end{equation}
\end{lem}
\begin{proof}
One can check that \eqref{eq:5} is equivalent to 
\begin{equation}
  \label{eq:4}
\det\left(\frac{x_j^{i-1}}{(ax_j)_i(q^{n-i}bx_j)_i}\right)_{i,j=1}^n
=\frac{\prod_{1\le i<j\le n} (1-abq^{n-1}x_ix_j)(x_j-x_i)}{\prod_{j=1}^n (ax_j)_n(bx_j)_n}.
\end{equation}
We will prove the following identity, which is equivalent to \eqref{eq:4}:
\begin{equation}
  \label{eq:2}
\det\left(x_j^{i-1} (aq^ix_j)_{n-i}(bx_j)_{n-i}\right)_{i,j=1}^n
=\prod_{1\le i<j\le n} (1-abq^{n-1}x_ix_j)(x_j-x_i).
\end{equation}
Let $f(x)$ (resp.~$g(x)$) be the left (resp.~right) hand side of \eqref{eq:2}. Since $f(x)=0$ whenever $x_i=x_j$ for $i\ne j$, it has $\prod_{1\le i<j\le n}(x_j-x_i)$ as a factor. Moreover, if $x_i=1/abq^{n-1}x_j$, one can check that the $i$th column is equal to the $j$th column multiplied by $(ab)^{1-n}q^{-(n-1)^2}x_j^2$. Thus $f(x)$ is a multiple of $g(x)$.
Since $\deg f(x)=\deg g(x)=3\binom n2$ and their coefficients of $x_1^{0}x_2^{1}\dots x_n^{n-1}$ are equal to $1$, we obtain $f(x)=g(x)$. 
\end{proof}

We note that Lemma~\ref{lem:det} is equivalent to \cite[Theorem~27]{KratDet}.

\begin{proof}[Proof of Theorem~\ref{thm:NL}]
By \eqref{eq:N_lambda,mu} and \eqref{eq:nu},
  \begin{align*}
&  N^L_\lambda(n;-uv,-u/v;q) = \det\left( \nu_{\lambda_i+n-i,\mu_j+n-j}^L(-uv,-u/v;q) \right)_{i,j=1}^n\\
&=\det\left((-1)^{\lambda_i-\mu_j-i+j}
q^{\binom{\lambda_j-j+i}2}\qbinom{n-j+\lambda_j}{\lambda_j-j+i}
\frac{(-uvq^{n-i+1})_{\lambda_j-j+1}}{(u^2q^{2n+1-i+\lambda_j-j})_{\lambda_j-j+i}}\right) _{i,j=1}^n\\
&=(-1)^{|\lm|} \prod_{i=1}^n 
\frac{q^{\binom{\lambda_i-i}2+\binom i2} (q)_{n-i+\lambda_i}(-uvq^{n-i+1})_{\lambda_i}}
{(q)_{\lambda_i-i}(q)_{n-i}(u^2q^{2n+1+\lambda_i-i})_{\lambda_i-i}}
\det\left( \frac{q^{i(\lambda_j-j)}}{(q^{\lambda_j-j+1})_i(u^2q^{2n+1-i+\lambda_j-j})_i} \right) _{i,j=1}^n.
  \end{align*}
Using Lemma~\ref{lem:det} with $x_j=q^{\lambda_j-j}$, we obtain the desired formula.
\end{proof}

Note that our proofs of Theorem~\ref{thm:ML} and Theorem~\ref{thm:NL} are
different in nature. We used a Selberg-type integral to prove
Theorem~\ref{thm:ML} and a determinant evaluation to prove Theorem~\ref{thm:NL}.
It is natural to ask for a different proof of Theorem~\ref{thm:ML} by evaluating
a determinant. This is indeed possible. Converting $M_\lambda^L(n;a,b;q)$ as a
determinant, one can check that Theorem~\ref{thm:ML} is equivalent to the
following result whose direct proof was found by Krattenthaler
\cite[p.~30]{KratDetBanff}.

\begin{prop}
We have
\[
\det\left(\frac{x_j^i(b/x_j)_i}{(ax_j)_i}\right)_{i,j=1}^n
=\prod_{i=1}^n \frac{(abq^i)_{i-1}(x_i-b)}{(ax_i)_n}
\prod_{1\le i<j\le n} (x_j-x_i).
\]
\end{prop}

\section{Enumeration of lecture hall tableaux of other types}
\label{sec:q-enum-lect}

In Sections~\ref{sec:lect-hall-tabl} and \ref{sec:moments-dual-moments}, we obtained some enumeration results for lecture hall tableaux of types $(n,\ge,>)$ and $(n,<,\le)$. In this section we prove similar results for lecture hall tableaux of types $(n,>,\ge)$ and $(n,\le,<)$.

Let $T\in \LHT_{(n,\prec_1,\prec_2)}(\lm)$, where $\prec_1$ and $\prec_2$ are fixed inequalities. Recall that $\wt(T)$ is defined by
\[
\wt(T)=\prod_{s\in \lm} y_{T(s)} u^{\flr{T(s)/(n+c(s))}} v^{o(\flr{T(s)/(n+c(s))})}.
\]
We also define $\overline{\wt}(T)$ by
\[
\overline\wt(T)=\prod_{s\in \lm} y_{T(s)} u^{\ceiling{T(s)/(n+c(s))}} v^{o(\ceiling{T(s)/(n+c(s))})}.
\]
For example, if $T$ is the lecture hall tableau in Figure~\ref{fig:LHT}, 
\begin{align*}
\wt(T)&=y_0^3y_1^3y_2^2y_3^2y_4^2y_5y_6y_9 u^3 v^3,\\
\overline{\wt}(T)&=y_0^3y_1^3y_2^2y_3^2y_4^2y_5y_6y_9 u^{13} v^{11}.
\end{align*}

Recall $h_k^{(n)}(\vec y;u,v)$ and $e_k^{(n)}(\vec y;u,v)$ in Definition~\ref{defn:he}, where
 $\vec y=(y_0,y_1,\dots)$ is a sequence of variables. 
 \begin{defn}
We define
\[
\overline{h}_k^{(n)}(\vec y;u,v) = \sum_{\alpha} \vec y_{\alpha} u^{|\ceiling{\alpha}_S|} v^{o(\ceiling{\alpha}_S)},
\]
where $S=(n,n+1,\dots,n+k-1)$ and the sum is over all sequences $\alpha=(\alpha_1,\dots,\alpha_k)$ of integers satisfying
\[
\frac{\alpha_1}{n}> \frac{\alpha_2}{n+1}> \dots > \frac{\alpha_k}{n+k-1}>0.
\]
We define
\[
\overline{e}_k^{(n)}(\vec y;u,v) =\sum_{\lambda} \vec y_{\lambda} u^{|\ceiling{\lambda}_S|} v^{o(\ceiling{\lambda}_S)},
\]
where $S=(n,n-1,\dots,n-k+1)$ and the sum is over all sequences $\lambda=(\lambda_1,\dots,\lambda_k)$ of integers satisfying
\[
\frac{\lambda_1}{n}\ge\frac{\lambda_2}{n-1}\ge \dots \ge \frac{\lambda_k}{n-k+1}>0.
\]
 \end{defn}

By definition, we have
\[
\overline{h}_k^{(n)}(\vec y;u,v) = \sum_{\alpha\in \overline{AL}_{n+k-1,k}} \vec y_\alpha u^{|\ceiling{\alpha}|} v^{o(\ceiling{\alpha})}, \qquad
\overline{e}_k^{(n)}(\vec y;u,v) =  \sum_{\lambda\in \overline{L}_{n,k}} \vec y_\lambda u^{|\ceiling{\lambda}|} v^{o(\ceiling{\lambda})}.
\]
Observe that the variable $y_0$ in $\vec y$ is never used in $\overline{h}_k^{(n)}(\vec y;u,v)$  and $\overline{e}_k^{(n)}(\vec y;u,v)$. 

For a sequence of variables $\vec y=(y_0,y_1,\dots)$, we define
$\vec y^+=(y^+_0,y^+_1,\dots)$, where $y^+_i=y_{i+1}$. In fact,
$\vec y^+$ is the same as $(y_1,y_2,\dots)$. However, in order to emphasis that
the index of the sequence begins with $0$ we use $\vec y^+=(y^+_0,y^+_1,\dots)$.
The following lemma is an immediate consequence of the map $\lambda\mapsto\lambda^+$
in Proposition~\ref{prop:+-}.

\begin{lem}\label{lem:hh'ee'}
Let $\vec y=(y_0,y_1,\dots)$ be a sequence of variables. Then for integers $n\ge k\ge0$, we have
  \begin{align*}
\overline{h}^{(n)}_k(\vec y;u,v)   &= u^kv^k h^{(n)}_k(\vec y^+;u,v^{-1}),\\
\overline{e}^{(n)}_k(\vec y;u,v)   &= u^kv^{k} e^{(n)}_k(\vec y^+;u,v^{-1}).
  \end{align*}
\end{lem}

Let $\prec_1$ and $\prec_2$ be any inequalities in $\{>,<,\ge,\le\}$. Recall that lecture hall tableaux in $\LHT_\lm^{(n,\prec_1,\prec_2)}$ may have entries equal to $0$.
We define $\overline{\LHT}_\lm^{(n,\prec_1,\prec_2)}$ to be the set of 
lecture hall tableaux in $\LHT_\lm^{(n,\prec_1,\prec_2)}$ all of whose entries are positive. 
We also define
\[
\overline{LS}_\lm^{(n,\prec_1,\prec_2)}(\vec y;u,v) = \sum_{T\in \overline{\LHT}_{(n,\prec_1,\prec_2)}(\lm)} \overline{\wt}(T).
\]
Note that
\[
\overline{h}_k^{(n)}(\vec y;u,v) = \overline{LS}^{(n,>,\ge)}_{(k)}(\vec y;u,v), \qquad \overline{e}_k^{(n)}(\vec y;u,v) = \overline{LS}^{(n,>,\ge)}_{(1^k)}(\vec y;u,v).
\]

For any lecture hall tableau $T$, we define $T^+$ to be the tableau obtained from $T$ by increasing every entry by $1$. By Lemma~\ref{lem:simple}, the map $T\mapsto T^+$ gives a bijection from $\LHT_\lm^{(n,\ge,>)}$ to $\overline{\LHT}_\lm^{(n,>,\ge)}$ and
a bijection  from $\LHT_\lm^{(n,<,\le)}$ to $\overline{\LHT}_\lm^{(n,\le,<)}$. Therefore we obtain
a relation between lecture hall Schur functions as follows. 

\begin{prop}\label{prop:LHT+-}
Let $\lambda$ and $\mu$ be partitions with $\ell(\lambda)\le n$ and $\mu\subseteq\lambda$.  Then
  \begin{align*}
\overline{LS}_\lm^{(n,>,\ge)}(\vec y;u,v) &=(uv)^{|\lm|}{LS}_\lm^{(n,\ge,>)}(\vec y^+;u,v^{-1}),\\
\overline{LS}_\lm^{(n,\le,<)}(\vec y;u,v) &=(uv)^{|\lm|}{LS}_\lm^{(n,<,\le)}(\vec y^+;u,v^{-1}).
  \end{align*}
\end{prop}

Recall that in Theorems~\ref{thm:schur} and \ref{thm:schur2} we have Jacobi--Trudi type formulas for $LS_\lm^{(n,\ge,>)}(\vec y;u,v)$ and $LS_\lm^{(n,<,\le)}(\vec y;u,v)$.
Combining these results with Lemma~\ref{lem:hh'ee'} and Proposition~\ref{prop:LHT+-},
we obtain Jacobi--Trudi type formulas for $\overline{LS}_\lm^{(n,>,\ge)}(\vec y;u,v)$ and $\overline{LS}_\lm^{(n,\le,<)}(\vec y;u,v)$.

\begin{thm}
Let $\lambda$ and $\mu$ be partitions with $\ell(\lambda)\le n$ and $\mu\subseteq\lambda$.  Then
\begin{align*}
\overline{LS}_\lm^{(n,>,\ge)}(\vec y;u,v)  
&= \det\left(\overline{h}_{\lambda_i-\mu_j-i+j}^{(n-j+1+\mu_j)}(\vec y;u,v)\right)_{i,j=1}^{\ell(\lambda)}
= \det\left(\overline{e}_{\lambda'_i-\mu'_j-i+j}^{(n+j-1-\mu_j')}(\vec y;u,v)\right)_{i,j=1}^{\ell(\lambda')},\\
\overline{LS}_\lm^{(n,\le,<)}(\vec y;u,v) &= \det\left(\overline{e}_{\lambda_i-\mu_j-i+j}^{(n-i+\lambda_i)}(\vec y;u,v)\right)_{i,j=1}^{\ell(\lambda)} 
= \det\left( \overline{h}_{\lambda'_i-\mu'_j-i+j}^{(n+i-\lambda'_i)}(\vec y;u,v)\right)_{i,j=1}^{\ell(\lambda')}.
\end{align*}
\end{thm}

In Corollaries~\ref{cor:LHT_prod} and \ref{cor:LHT_prod2}, we have
product formulas for $LS^{(n,\ge,>)}_{\lambda}(\vec q;u,v)$ and $LS^{(n,<,\le)}_{\lambda}(\vec q;u,v)$. 
By Proposition~\ref{prop:LHT+-}, we obtain product formulas for
$\overline{LS}^{(n,>, \ge)}_{\lambda}(\vec q;u,v)$ and $\overline{LS}^{(n,\le,<)}_{\lambda}(\vec q;u,v)$. 

\begin{thm}
For a partition $\lambda$ with $\ell(\lambda)\le n$, 
  \begin{align*}
\overline{LS}_\lambda^{(n,>,\ge)}(\vec q;u,v) &= (uvq)^{|\lambda|} \prod_{1\le i<j\le n} \frac{q^{\lambda_j+n-j}-q^{\lambda_i+n-i}}{q^{i-1}-q^{j-1}}  
\prod_{i=1}^n \frac{(-uv^{-1}q^{n-i+1})_{\lambda_i}}{(u^2q^{2n-i+1})_{\lambda_i}},\\
\overline{LS}_\lambda^{(n,\le,<)}(\vec q;u,v) &= (uvq)^{|\lambda|} q^{n(\lambda')-n(\lambda)} \prod_{1\le i<j\le n} \frac{q^{\lambda_j+n-j}-q^{\lambda_i+n-i}}{q^{i-1}-q^{j-1}}  \\
\notag & \qquad \qquad \times 
\prod_{i=1}^n \frac{(-uv^{-1}q^{n-i+1})_{\lambda_i}}{(u^2q^{n-i+1+\lambda_i})_{n-i+\lambda_i}}
\prod_{1\le i<j\le n}(1-u^2q^{2n+\lambda_i+\lambda_j-i-j+1}).
  \end{align*}
\end{thm}

\section{Further study}
\label{sec:concl-furth-plans}

A lot of natural questions arise from these lecture hall tableaux. We list a few of them
in this section.
\begin{enumerate}
\item In this paper, we study the case $q=t$ of the multivariate little
  $q$-Jacobi polynomials. Stokman defined the $(q,t)$-analog in
  \cite{Stokman97}. Do we get nice combinatorics if we set $t=q^k$? Can we use
  the $t=q^k$ version of Warnaar's $q$-Selberg integral in this setting
  \cite{Warnaar2005}? Another natural problem is to generalize our results to
  multivariate big $q$-Jacobi polynomials \cite{Olshanski}. 

\item There is a lot of recent activities around the enumeration of skew (semi-)standard
Young tableaux \cite{MPP}. Naruse \cite{Naruse} found a subtraction-free formula for the number of standard Young tableaux of shape $\lm$.
Morales, Pak and Panova \cite{MPP} proved the following $q$-analog of Naruse's result:
\[
s_{\lambda/\mu}(1, q, q^2,\ldots)=\sum_{S\in \mathcal{E}(\lambda/\mu)}
\prod_{(i,j)\in\lambda\backslash S} \frac{q^{\lambda'_j-i}}{1-q^{h(i,j)}},
\]
where $s_{\lambda/\mu}(1, q, q^2,\ldots)$ is the principal specialization of the skew Schur function,
the elements in $\mathcal{E}_{\lambda/\mu}$ are certain subsets of $\lm$ called excited diagrams and $h(i,j)=\lambda_i-i+\lambda'_j-j+1$. 
In our setting, when $n\rightarrow\infty$, both $LS_{\lambda/\mu}^{(n,\ge,>)}(\vec y; u,v)$
and $LS_{\lambda/\mu}^{(n,<,\le)}(\vec y; u,v)$ converge to the principal specialization of a skew Schur function.
It is therefore natural to ask whether there exist Naruse-type formulas for $LS_{\lambda/\mu}^{(n,\ge,>)}(\vec y; u,v)$
and $LS_{\lambda/\mu}^{(n,<,\le)}(\vec y; u,v)$.

\item The lecture hall tableaux define lecture hall tableau polytopes. When $n\rightarrow\infty$, these
polytopes are the Gelfand-Tsetlin polytopes, which are known to have nice properties \cite{DM}.
 The lecture hall polytopes were also studied by different authors. For a survey of those results,
see \cite{LHPSavage}. Do the  lecture hall tableau polytopes inherit nice properties of
these families of polytopes?

\item Given a sequence $a=(a_1,a_2,\ldots)$, an integer $n\ge0$ and a partition $\lambda$, the $a$-lecture hall tableaux can be defined as 
fillings $T$ of the diagram of $\lambda$
such that
\[
\frac{T_{i,j}}{a_{n+j-i}}\ge \frac{T_{i,j+1}}{a_{n+j+1-i}}; \ \ 
\frac{T_{i,j}}{a_{n+j-i}}\ge \frac{T_{i+1,j}}{a_{n+j-1-i}}.
\]
Here we study the case $a=(1,2,3,\ldots)$. 
Are there natural sequences? In the case of lecture hall partitions, Bousquet-M\'elou and
Eriksson \cite{BME2} showed  that, for example, given $\ell\ge 2$,  the sequence with  $a_i=\ell a_{i-1}-a_{i-2}$
for all $i$ gives beautiful generating functions. 
Savage and Visontai \cite{SV} studied $a$-Eulerian polynomials coming from $a$-lecture hall partitions.
Can we build a tableau analog of $a$-Eulerian polynomials?

\item In \cite{Viennot} Viennot found a combinatorial interpretation for
  $\sigma_{n,k,\ell}:=\frac{\LL(x^np_kp_\ell)}{\LL(p_\ell^2)}$ in terms of Motzkin
  paths. Since
\[
x^n = \sum_{k=0}^n \sigma_{n,k} p_k(x),\qquad p_n(x) = \sum_{k=0}^n \nu_{n,k} x^k,
\]
we have
\[
x^np_k= \sum_{i=0}^k \nu_{k,i} x^{n+i} = \sum_{i=0}^k\nu_{k,i} \sum_{\ell=0}^{n+i} \sigma_{n+i,\ell} \cdot p_\ell.
\]
Multiplying both sides by $p_\ell$ and taking $\LL$, we obtain
\[
\sigma_{n,k,\ell} = \sum_{i=\max(n-\ell,0)}^k \nu_{k,i} \sigma_{n+i,\ell}.
\]
Note that $\sigma_{0,k,\ell}=\delta_{k,\ell}$ is the orthogonality of the polynomials, whose general version for the little $q$-Jacobi polynomials is proved combinatorially in Proposition~\ref{prop:eh}.
Is there any nice combinatorial interpretation for $\sigma_{n,k,\ell}$ for the little $q$-Jacobi polynomials in terms of lecture hall partitions? Is there a multivariate analog?

\item Our lecture hall Schur functions $LS_{\lambda/\mu}^{(n,<,\le)}(\vec y; u,v)$
and $LS_{\lambda/\mu}^{(n,\ge,>)}(\vec y; u,v)$ are symmetric functions in $\vec y$ if $u=v=0$
or if $n\rightarrow \infty$. In general they are not symmetric. We leave as an open problem
to study the algebraic properties of these functions. 

\item Last but not least, this phenomena that multivariate (dual) moments give
  rise to interesting combinatorics seems to be a ``universal'' phenomena. The
  moments of the (modified) Al-Salam--Chihara polynomials are related to a
  certain exclusion process with open boundaries and give rise to some nice
  combinatorial interpretation in terms of rhombic alternative tableaux
  \cite{Mandelshtam_2018}. Recently D. Kim \cite{kim20:combin_al_salam_chihar}
  gave a combinatorial interpretation of the coefficients of these polynomials.
  In \cite{CMW}, a combinatorial interpretation is found for the multivariate
  moments of the Koornwinder polynomials with $q=t$ for a specific $\lambda$,
  which gives a positivity result for the Koornwinder moments. Corteel and
  Williams \cite{CW_Koor} conjectured that the positivity is true for general
  $\lambda$. It would be very interesting to develop a general combinatorics
  theory for multivariate moments and coefficients of orthogonal polynomials.
\end{enumerate}

\section*{Acknowledgments}
The authors want to thank the University of California, Berkeley where most of
this work was completed. They are grateful to Ole Warnaar for his helpful
comments. S.C. is supported by the MSRI (NSF grant DMS-0932078) and the Miller
Institute during her sabbatical at UC Berkeley. J.S.K. was supported by NRF
grants \#2019R1F1A1059081 and \#2016R1A5A1008055.


\begin{thebibliography}{10}

\bibitem{Aomoto1987}
K.~Aomoto.
\newblock Jacobi polynomials associated with Selberg integrals.
\newblock {\em SIAM Journal on Mathematical Analysis}, 18(2):545--549, 1987.

\bibitem{BME1} 
M. Bousquet-M\'elou and K. Eriksson.
\newblock Lecture hall partitions.
\newblock {\em The Ramanujan Journal}, 1 no. 1 (1997):101--111.

\bibitem{BME2} 
M. Bousquet-M\'elou and K. Eriksson.
\newblock Lecture hall partitions II.
\newblock {\em The Ramanujan Journal}, 1 no. 2 (1997):165--186.

\bibitem{BL}
P.~Br\"and\'en and M.~Leander, 
\newblock Lecture hall $P$-partitions.
\newblock {\em Journal of Combinatorics}, 11(2):391--412, 2020.

\bibitem{CKS}
S. Corteel, J.S. Kim and D. Stanton.
\newblock Moments of orthogonal polynomials and combinatorics.
\newblock  {\em Recent Trends in Combinatorics}, The IMA Volumes in Mathematics and its Applications 159, 2016.

\bibitem{CMW}
S.~Corteel, O.~Mandelshtam, and L.~Williams.
\newblock Combinatorics of the two-species {ASEP} and {K}oornwinder moments.
\newblock {\em Adv. Math.}, 321:160--204, 2017.

\bibitem{ALHC}
S.~Corteel and C.~D. Savage.
\newblock Anti-lecture hall compositions.
\newblock {\em Discrete Math.}, 263(1-3):275--280, 2003.

\bibitem{trunc_LHP}
S.~Corteel and C.~D. Savage.
\newblock Lecture hall theorems, {$q$}-series and truncated objects.
\newblock {\em J. Combin. Theory Ser. A}, 108(2):217--245, 2004.

\bibitem{CW_Koor}
S.~Corteel and L.~Williams.
\newblock Macdonald--Koornwinder moments and the two-species exclusion process.
\newblock {\em Sel. Math. New Ser.} (2018) 24:2275--2317. 

\bibitem{DM}
J. A. De Loera and T. B. McAllister.
\newblock Vertices of Gelfand-Tsetlin Polytopes.
\newblock {\em Discrete and Computational Geometry}, 32 (2004), no. 4, 459--470.

\bibitem{Eriksson_1998}
H.~Eriksson and K.~Eriksson.
\newblock {Affine Weyl groups as infinite permutations}.
\newblock {\em The Electronic Journal of Combinatorics}, 5(1), 1998.

\bibitem{Forrester2008}
P.~J. Forrester and S.~O. Warnaar.
\newblock The importance of the {S}elberg integral.
\newblock {\em Bull. Amer. Math. Soc. (N.S.)}, 45(4):489--534, 2008.

\bibitem{GR}
G.~Gasper and M.~Rahman.
\newblock {\em Basic hypergeometric series}, volume~96 of {\em Encyclopedia of
  Mathematics and its Applications}.
\newblock Cambridge University Press, Cambridge, second edition, 2004.
\newblock With a foreword by Richard Askey.

\bibitem{Kadell1988a}
K.~W.~J. Kadell.
\newblock A proof of some {$q$}-analogues of {S}elberg's integral for {$k=1$}.
\newblock {\em SIAM J. Math. Anal.}, 19(4):944--968, 1988.

\bibitem{kim20:combin_al_salam_chihar}
D.~Kim.
\newblock {Combinatorial formulas for the coefficients of the Al-Salam-Chihara
  polynomials}.
\newblock {\it Preprint},
  \href{https://arxiv.org/abs/2002.01518v1}{arXiv:2002.01518v1}.

\bibitem{KLS}
R.~Koekoek, P.~A. Lesky, and R.~F. Swarttouw.
\newblock {\em Hypergeometric orthogonal polynomials and their
  {$q$}-analogues}.
\newblock Springer Monographs in Mathematics. Springer-Verlag, Berlin, 2010.
\newblock With a foreword by Tom H. Koornwinder.

\bibitem{Koornwinder92}
T.~H. Koornwinder.
\newblock Askey-{W}ilson polynomials for root systems of type {$BC$}.
\newblock In {\em Hypergeometric functions on domains of positivity, {J}ack
  polynomials, and applications ({T}ampa, {FL}, 1991)}, volume 138 of {\em
  Contemp. Math.}, pages 189--204. Amer. Math. Soc., Providence, RI, 1992.

\bibitem{KratDet}
C.~Krattenthaler.
\newblock Advanced determinant calculus.
\newblock {\em S{\'e}m. Lothar. Combin.}, 42:Art. B42q, 67 pp. (electronic),
  1999.
\newblock The Andrews Festschrift (Maratea, 1998).

\bibitem{KratDetBanff}
C.~Krattenthaler.
\newblock Advanced determinant calculus.
\newblock BIRS Workshop: Asymptotic Algebraic Combinatorics, 19w5220, 2019.
\newblock \url{https://www.birs.ca/workshops/2019/19w5220/files/Krattenthaler.pdf}

\bibitem{Luque2003}
J.-G. Luque and J.-Y. Thibon.
\newblock Hankel hyperdeterminants and {S}elberg integrals.
\newblock {\em Journal of Physics A: mathematical and general}, 36(19):52--67,
  2003.

\bibitem{Macdonald_Schur}
I.~G. Macdonald.
\newblock Schur functions: theme and variations.
\newblock In {\em S{\'e}minaire {L}otharingien de {C}ombinatoire
  ({S}aint-{N}abor, 1992)}, volume 498 of {\em Publ. Inst. Rech. Math. Av.},
  pages 5--39. Univ. Louis Pasteur, Strasbourg, 1992.
  
\bibitem{Macdonald}
I.~G. Macdonald.
\newblock {\em Symmetric functions and {H}all polynomials}.
\newblock Oxford Mathematical Monographs. The Clarendon Press Oxford University
  Press, New York, second edition, 1995.
\newblock With contributions by A. Zelevinsky, Oxford Science Publications.

\bibitem{Mandelshtam_2018}
O.~Mandelshtam and X.~Viennot.
\newblock {Tableaux combinatorics of the two-species PASEP}.
\newblock {\em Journal of Combinatorial Theory, Series A}, 159:215–239, 2018.
  
\bibitem{Mimachi1998}
K.~Mimachi.
\newblock The little $q$-Jacobi polynomial associated with a $q$-Selberg integral.
\newblock {\em Funkcialaj Ekvacioj}, 41(1):91--100, 1998.

\bibitem{MPP}
A.~Morales, I.~Pak, and G.~Panova.
\newblock Hook formulas for skew shapes {I}. $q$-analogues and bijections.
\newblock {\em Journal of Combinatorial Theory, Series A}, 154:350--405, 2018.

\bibitem{NNSY}
J. Nakagawa, M. Noumi, M. Shirakawa, Y. Yamada.
\newblock Tableau representation for Macdonald’s ninth variation of Schur functions.
\newblock {\em Physics and Combinatorics}, 2000:180--195 (2001).

\bibitem{Naruse}
H.~Naruse.
\newblock Schubert calculus and hook formula.
\newblock Talk slides at 73rd S\'em. Lothar. Combin., Strobl, Austria, 2014;
  available at
  \url{https://www.emis.de/journals/SLC/wpapers/s73vortrag/naruse.pdf}.

\bibitem{Olshanski}
G.~I. Olshanski.
\newblock An analogue of big {$q$}-{J}acobi polynomials in the algebra of
  symmetric functions.
\newblock {\em Functional Analysis and Its Applications}, 51(3):204--220, 2017.

\bibitem{Olshanski_Selberg1}
G.~Olshanski and A.~Osinenko.
\newblock Multivariate {J}acobi polynomials and the {S}elberg integral.
\newblock {\em Funct. Anal. Appl.}, 46(4):262--278, 2012.

\bibitem{Olshanski_Selberg2}
G.~Olshanski and A.~Osinenko.
\newblock Multivariate {J}acobi polynomials and the {S}elberg integral. {II}.
\newblock {\em J. Math. Sci. (N.Y.)}, 215(6):755--768, 2016.

\bibitem{LHPSavage}
C.~D. Savage.
\newblock The mathematics of lecture hall partitions.
\newblock {\em J. Combin. Theory Ser. A}, 144:443--475, 2016.

\bibitem{SV}   
C. D. Savage and M. Visontai.
\newblock The $s$-Eulerian polynomials have only real roots.
\newblock {\em Trans. Amer. Math. Soc.}, 367(2015), no. 2, 1441--1466.

\bibitem{Stokman97}
J.~V. Stokman.
\newblock Multivariable big and little {$q$}-{J}acobi polynomials.
\newblock {\em SIAM J. Math. Anal.}, 28(2):452--480, 1997.

\bibitem{StokmanKoornwinder97}
J.~V. Stokman and T.~H. Koornwinder.
\newblock Limit transitions for {BC} type multivariable orthogonal polynomials.
\newblock {\em Canad. J. Math.}, 49(2):373--404, 1997.

\bibitem{Viennot}
X.~Viennot, Une th\'eorie combinatoire des polyn\^ omes orthogonaux, Lecture Notes UQAM, 219p., Publication du LACIM, Universit\'e du Qu\'ebec \`a Montr\'eal, 1984, r\'eed. 1991.

\bibitem{Warnaar2005}
S.~O. Warnaar.
\newblock {$q$}-{S}elberg integrals and {M}acdonald polynomials.
\newblock {\em Ramanujan J.}, 10(2):237--268, 2005.

\end{thebibliography}
\end{document}